\documentclass[12pt]{article}
\usepackage{amsmath,amssymb,latexsym}
\usepackage{enumerate}
\usepackage{amsthm}
\usepackage{graphicx}
\usepackage{epsfig}
\usepackage[export]{adjustbox}
\usepackage{lipsum}
\usepackage{paralist}
\usepackage{eucal}
\usepackage{enumitem}

\usepackage{amsthm,amsmath,amsfonts}
\usepackage{MnSymbol}
\usepackage{mathtools}

\usepackage{float}

\usepackage[colorlinks=true,
            linkcolor=red]
           {hyperref}

\usepackage[font=footnotesize]{caption}
\usepackage[margin=1in]{geometry}
\usepackage[singlespacing]{setspace}
\usepackage{mathdots}

\usepackage{comment}

\usepackage[dvipsnames]{xcolor}
\usepackage{color}
\usepackage{pgfplots}

\usepackage{dsfont}

\usepackage{tikz}
\usepackage{forest}

\usetikzlibrary{lindenmayersystems}
\usetikzlibrary{positioning,automata}

\newcommand\tmu{\tilde{\mu}}
\newcommand\Xip{\Xi_{\Phi}}

\newcommand\NN{\mathbb{N}}
\newcommand\RR{\mathbb{R}}

\newcommand\AL{\mathcal{A}}
\newcommand\MM{\mathcal{M}}

\newcommand\GG{\mathcal{G}}

\newcommand\II{\mathcal{I}}

\newcommand\LL{\mathcal{L}}

\newcommand\OO{\mathcal{O}}
\newcommand\UU{\mathcal{U}}

\newcommand\Zt{\tilde{Z}}

\newcommand{\nlt}{\not\kern-0.4em{\lhd_{\mathfrak{t}}}}

\newcommand{\diam}{\text{diam}}

\newcommand{\infom}[1][]{I_{\mu}}

\newcommand{\dom}{\text{dom}}

\newcommand*\om[0]{\omega}

\DeclareMathOperator*{\osc}{osc}
\DeclareMathOperator*{\HD}{HD}

\DeclareMathOperator*\sgn{sgn}

\definecolor{MyGreen}{RGB}{36, 120, 35}

\newtheorem{theorem}{Theorem}[section]
\numberwithin{theorem}{subsection}
\newtheorem{definition}[theorem]{Definition}

\newtheorem{corollary}[theorem]{Corollary}

\newtheorem{example}[theorem]{Example}
\newtheorem{lemma}[theorem]{Lemma}
\newtheorem{remark}[theorem]{Remark}
\newtheorem{proposition}[theorem]{Proposition}

\newtheorem{innercustomprop}{Proposition}
\newenvironment{customprop}[1]
  {\renewcommand\theinnercustomprop{#1}\innercustomprop}
  {\endinnercustomthm}

\newtheorem{innercustomthm}{Theorem}
\newenvironment{customtheorem}[1]
  {\renewcommand\theinnercustomthm{#1}\innercustomthm}
  {\endinnercustomthm}

\numberwithin{equation}{section}

\newcommand{\address}{}

\providecommand{\phantomsection}{}
\AtBeginDocument{\let\textlabel\label}
\makeatletter
\newcommand{\mylabel}[2]{\raisebox{.7\normalbaselineskip}{\phantomsection}(#1)
	\def\@currentlabel{#1}\textlabel{#2}}
\makeatother

\makeatletter
\newcommand\xlabel[2][]{\phantomsection\def\@currentlabelname{#1}\label{#2}}
\makeatother

\newcommand\blfootnote[1]{%
  \begingroup
  \renewcommand\thefootnote{}\footnote{#1}%
  \addtocounter{footnote}{-1}%
  \endgroup
}

\renewcommand{\address}{Address: Department of Mathematics, University of North Texas, 1155 Union Circle \#311430, Denton, TX 76203-5017, USA; E-mail: NathanDalaklis@my.unt.edu}

\title{Multifractal Analysis of F-Exponents for Finitely Irreducible Conformal Graph Directed Markov Systems}
\author{Nathan Dalaklis \\University of North Texas \footnote{\address}}

\pgfplotsset{compat=1.18}
\begin{document}

\maketitle

\begin{abstract}
Consider a conformal graph directed Markov system (CGDMS) with a finitely irreducible symbolic representation over a countable alphabet and its corresponding limit set. Under a mild condition on the system, we give a multifractal analysis of level sets of Birkhoff averages with respect to Hausdorff dimension for a large family of functions. In the process of the development of this multifractal theory for finitely irreducible CGDMS, we repair, extend, and generalize previous analysis done for the Gauss Map.  We conclude with application of these results to a few examples in the case of both finite and countably infinite alphabets.

\end{abstract}

\tableofcontents

\section{Introduction}

The study of Birkhoff averages of a $\mu$-measure preserving dynamical system may be viewed as the study of the exceptional sets to G.D. Birkhoff's Ergodic Theorem from 1931 \cite{birkhoffProofErgodicTheorem1931} with respect to $\mu$. A short and elegant proof of Birkhoff's Ergodic Theorem is given by Katok and Hasselblatt \cite[Theorem 4.1.2]{katokIntroductionModernTheory1995}. Versions of the statement can also be found in \cite[Theorem 8.2.11, Theorem 8.2.12, and Corollary 8.2.14]{urbanskiVolumeErgodicTheory2021}. Given a dynamical system $T:X\to X$ with limit set $J$ and $g:X\to \RR$ a H\"older potential or a bounded observable, a multifractal consideration of these averages concerns itself with sets of the form 
\[
    J(\xi) = \left\{x\in J \;\bigg|\; \lim_{n\to\infty}\frac{1}{n}\sum_{k=0}^{n-1} g ((T^{n}(x))= \xi \right\}
\]
and
\[
    J' = \left\{x\in J\;\bigg|\; \lim_{n\to\infty}\frac{1}{n}\sum_{k=0}^{n-1} g (T^n(x)) \text{ does not exist}\right\}
\]
 where $\xi\in (\xi_{\min}(g),\xi_{\max}(g))$ with one or both of $\xi_{\min}(g)$ and $\xi_{\max}(g)$ are potentially infinite. When collected together these level sets for which $J(\xi)\neq\emptyset$ along with the exceptional set form a partition of the limit set $J$ called the multifractal decomposition of $J$ by Birkhoff averages with respect to $g$. The resulting spectrum function $t: (\xi_{\min}(g),\xi_{\max}(g))\to [0,HD(J)]$ defined by 
\[
    t(\xi) = HD(J(\xi)) 
\]
is then studied. Such multifractal analyses have been preformed for a variety of particular maps and general scenarios. See \cite{barreiraFrequencyDigitsLuroth2009, fanKhintchineExponentsLyapunov2009, iommiMultifractalAnalysisBirkhoff2015, johanssonMultifractalAnalysisNonuniformly2010, pesinMultifractalAnalysisBirkhoff2001, rushMultifractalAnalysisMarkov2023} and related multifractal analyses in \cite{fanFrequencyPartialQuotients2010, hanusThermodynamicFormalismMultifractal2002, kessebohmerMultifractalAnalysisSternBrocot2007}. 

These multifractal analyses describe the geometric structure of Birkhoff averages by way of Hausdorff dimension. Such problems go back to Besicovitch \cite{besicovitchSumDigitsReal1935} who looked at base $2$ expansions for which the corresponding dynamical system is the doubling map $T:[0,1]\to[0,1]$. In this case, one may remove the dyadic rationals from the analysis as they have Hausdorff Dimension $0$, and so, we may think of $J$ as the set of dyadic irrationals. A similar conversion to the dynamical systems setting where a negligible set may be removed from the analysis can be formulated for other number theoretic maps like other whole-number-base expansions, L\"uroth expansions, and complex continued fractions among others. 

The analysis here collects together many of these ideas by studying these Birkhoff decompositions in the general setting of finitely irreducible conformal graph directed Markov systems (CGDMS) under certain properties. Though initially defined in \cite{mauldinDimensionsMeasuresInfinite1996} and further developed in \cite{mauldinGraphDirectedMarkov2003}, an updated (and quite detailed) introduction to such systems can be found in \cite{urbanskiVolumeFinerThermodynamic2022}. 
\begin{figure}
    \centering
    \includegraphics{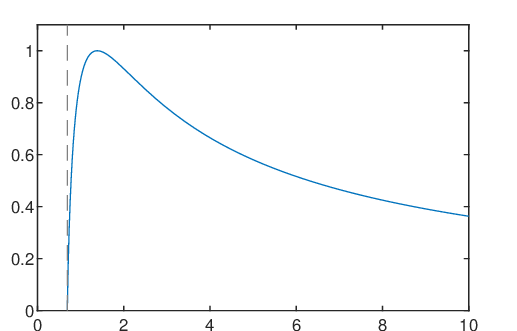}
    \caption{Graph of $\HD(J(\xi))$ for Example \ref{ex:Lu}.}
    \label{fig:Lu0To101}
    \end{figure}

\subsection{Main Results}
 
 The multifractal theory that follows generalizes, extends, and develops the method used by Fan et al. \cite{fanKhintchineExponentsLyapunov2009} which preformed the analysis for the family $-F = \{\log e\}_{e\in\mathbb{N}}$ and for which the corresponding exponent is the Khintchine exponent. The case of the Lyapunov exponent is addressed as well, but there are a few non-fatal issues with the argument which we repair in our general theory. Kesseb\"ohmer and Stratmann preformed such a multifractal analysis of the Gauss map prior to Fan et al. in \cite{kessebohmerMultifractalAnalysisSternBrocot2007}. The interested reader may also look to \cite{iommiMultifractalAnalysisBirkhoff2015} for further motivation for our approach. Our main theorem is as follows. (c.f \cite[Propositions 6.3 and 4.13, Theorems 1.2 and 1.3]{fanKhintchineExponentsLyapunov2009})

\begin{theorem}\label{thm:NDThm1}
Let $\Phi = \{\phi_e\}_{e\in\NN}$ be a cofinitely regular finitely irreducible CGDMS satisfying the SOSC. Let $F$ be a strictly positive H\"older family of potentials that are either comparable to $\log\Phi'$ (see  Lemma \ref{lem:Dopen}) or bounded. Let $D$ be the Manhattan region of the associated pressure function $P(t,q)$, and $(t,q)\in D$. If the amalgamated functions $\tilde{f}$ and $\Xip$ of $F$ and $\log\Phi'$ respectively are such that 
\begin{equation}
    \int_{E_A^{\infty}}(|\tilde{f}|+|\Xip|)d\tmu_{t,q}<\infty,
\end{equation}
where $\tmu_{t,q}$ is the unique shift invariant Borel probability measure equivalent to the unique $t\Xip+q\tilde{f}$-conformal probability measure $\tilde{m}_{t,q}$ on $E_A^{\infty}$ such that $\LL^{*}_{f_{t,q}}\tilde{m}_{t,q} = e^{P(t,q)}\tilde{m}_{t,q}$ (defined in \cite[Chapter 17.6]{urbanskiVolumeFinerThermodynamic2022}), then:
\begin{enumerate}
    \item The system of equations 
\begin{equation}
    \begin{cases}P(t,q) = q\xi \\ \frac{\partial P}{\partial q}(t,q) = \xi \end{cases} 
\end{equation}
has a unique solution $(t(\xi),q(\xi))\in D_0 = \{(t,q)\in D\;|\; 0\leq t\leq h\}$ for $\xi\in (\xi_{min},\infty)$ and $h$ the Bowen parameter of $\Phi$.
    \item $t(\xi)$ and $q(\xi)$ are real analytic.
    \item $t(\xi) = \HD(J(\xi))$. 
    \item $t(\xi) = \HD (J_1(\xi))$.
\end{enumerate}
\end{theorem}

In order to address such a general theorem, we will first concern ourselves with the properties of limiting equilibrium states analogous to the investigation of zero-temperature limits in \cite{jenkinsonZeroTemperatureLimits2005}. The following propositions generalize \cite{jenkinsonZeroTemperatureLimits2005} to the setting of multivariate families of equilibrium states and repairs the analysis of Fan et al. in \cite{fanKhintchineExponentsLyapunov2009}. They are also applicable to our setting. 

\begin{proposition}[c.f. \cite{jenkinsonZeroTemperatureLimits2005} Lemma 2] \label{MainResult2} Fix $0\leq t\leq h$ and let $q\in (-\infty,q_1)\subsetneq(-\infty,-1)$ such that $(0,q_1)\in D$. Then the families of Gibbs states $\{\{\tmu_{t,q}\}_{q< q_1}\}_{0\leq t\leq h}$ are tight, that is for all $\varepsilon>0$ there exists $K\subseteq E_A^{\infty}$ compact such that $\tmu_{t,q}(K)>1-\varepsilon$ for all $q<q_1$.
\end{proposition}

\begin{proposition}\label{MainResult3}
    Given the comparability of $F$ and $\log\Phi'$ from Lemma \ref{lem:Dopen}. For any $0\leq t\leq h$ fixed and any strictly decreasing sequence $\{q_n\}_{n\in\NN}$ for which $q_n\to -\infty$ as $n\to\infty$, there is a weak* accumulation point $\mu_{\infty}$ of the sequence of Gibbs states $\{\tmu_{t,q_n}\}_{n\in\NN}$, and this measure is $\tilde{f}$-minimizing, i.e. $\int_{E_A^{\infty}}\tilde{f}d\mu_{\infty} = \xi_{\min}$.
\end{proposition}

\subsection{Preliminaries for CGDMSs}
In most cases we will proceed in line with the definitions set in Urba\'nski, Roy, and Munday's  dynamics volumes, \cite{urbanskiVolumeErgodicTheory2021} and \cite{urbanskiVolumeFinerThermodynamic2022} with slight variation. We first collect together the lengthy definition of a CGDMS. 
\begin{definition}
    A \textbf{Conformal Graph Directed Markov System (CGDMS)} is a triple $(\Phi,\GG,A)$ where $\Phi =\{\phi_e: X_{t(e)}\to X_{i(e)}\}_{e\in E}$ is a collection of maps between a finite family of compact metric spaces $\{X_v\}_{v\in V}$, $V$ is the vertex set of $\GG = (V,E,i,t)$ a directed multi-graph, $A:E\times E\to\{0,1\}$ is a matrix with at least one $1$ in each row defining an allowable composition rule respecting the graph $\GG$ called $A$-admissibility. That is, $A_{ab} = 1$ if and only if the composition $\phi_{a}\circ\phi_{b}: X_{t(b)}\to X_{i(a)}$ is allowable and defined. Further the underlying Graph Directed System (GDS) defined by setting $A$ to be the adjacency matrix of the line graph of $\GG$ satisfies the following properties:
    \begin{enumerate}
        \item For every vertex $v\in V$, $X_v$ there is a $d(v)\geq 1$ such that $X_v$ is a compact subset of $\RR^{d(v)}$ which is regular closed.
        \item $\Phi$ satisfies the open set condition, that is for all $e_1,e_2\in E$, $e_1\neq e_2$ 
        \[
            \phi_{e_1}(\text{Int}(X_{t(e_1)}))\cap \phi_{e_2}(\text{Int}(X_{t(e_2)})) = \emptyset.
        \]
        \item For every $v\in V$ there is an open, connected and bounded set $W_v$ such that $X_v\subseteq W_v\subseteq \mathbb{R}^{d}$ and such that for every $e\in E$ with $t(e)=v$, the map $\phi_e$ extends to a contracting conformal diffeomorphism of $W_v$ into $\RR^{d}$, and all of these extensions have a common contraction ratio $s<1$. 
        \item There exists constants $L\geq 1$ and $\kappa>0$ such that 
        \[
            \left||\phi'_e(y)|-|\phi'_e(x)|\right|\leq L|\phi_e'(x)|\cdot \|y-x\|^{\kappa}
        \]
        for every $e\in E$ and every pair of points $x,y\in W_{t(e)}$.
    \end{enumerate}
\end{definition}
In other words $\Phi$ defines the inverted edges of $\GG$ and $A$ defines the traversable inverted paths through $\GG$. One should note that CGDMSs are often constructed by taking the inverse branches of a distance expanding system. Since $\Phi$ contains the information of the underlying GDS, in an abuse of notation we often refer to $\Phi$ alone as a CGDMS. In many contexts, it is helpful to also attach the symbolic representation $(E_A^{\infty},\sigma)$ to $\Phi$ where allowable concatenations of symbols are given by $A$ and $\sigma:E_A^{\infty}\to E_A^{\infty}$ is the one-sided shift map. For a finite word $w = w_1\dots w_n$ we define the length of the word $|w|$ to be the number of concatenated symbols that form $w$. The language of the symbolic representation $E_A^{*}$ is the set of all finite $A$-admissible words, 
\[
        E_A^{*} = \{\varepsilon\}\cup\left(\bigcup_{n=1}^{\infty} \{w\in E^n\;|\; A_{w_iw_{i+1}} = 1 \text{ for all } 1\leq i \leq n-1\}\right).
\]
One core property of systems we are interested in is defined by this $A$-admissibility rule.
\begin{definition}
    We say that a CGDMS $\Phi$ or its symbolic representation $E_A^{\infty}$ are \textbf{finitely irreducible} if the $A$-admissibility rule is such that there exists $I\subseteq E_A^{*}$ where $|I|<\infty$ and for every $e,f\in E$ there exists $\tau(e,f)\in I$ such that $e\tau(e,f)f\in E_A^{*}$. 
\end{definition}
 It is useful to note that if a CGDMS is finitely irreducible, then there exists $\tau\in I$ such that there are infinitely many letters $e\in E$ for which one may choose $\tau(e,e) = \tau$. This follows immediately from the pigeonhole principal. To see how the symbolic representation encodes a CGDMS, $\Phi$, we note that $w\in E_A^*$ represents a composition of maps in $\Phi$ with the exception that $\varepsilon\in E_A^0$, the empty word, is the unique word for which $w\varepsilon = \varepsilon w = w$ and for which $\phi_{\varepsilon} = \text{Id}_X: \bigsqcup_{v\in V} X_v \to \bigsqcup_{v\in V} X_v $ is the identity map on the disjoint union of the $X_v$'s. For all other $w$'s we have that
\[
    w\mapsto \phi_{w} := \phi_{w_1}\circ \phi_{w_2}\circ\cdots\circ\phi_{w_n}:X_{t(w_n)}=:X_{t(w)}\to X_{i(w)}:= X_{i(w_1)}.
\]
To limit this composition process, for each $n\in\NN$ we define the map $\cdot\vert_n:E_A^{\infty}\to E_A^n$ by $\om\vert_n := \om_1\dots \om_n$ for all $\om\in E_A^{\infty}$. We note that $\{\phi_{\om|_{n}}(X_{t(\om|_n}))\}_{n=0}^{\infty}$ is a descending sequence of compact sets, and the coding map $\pi:E_A^{\infty}\to X$ is therefore defined by 
\[
    \{\pi(\om)\} = \bigcap_{n=1}^{\infty} \phi_{\om\vert_n}(X_{t(\om\vert_n)}).
\]
The limit set of a CGDMS is given by 
\[
    J:= \pi(E_A^{\infty}) = \bigcup_{\om\in E_A^{\infty}}\{\pi(\om)\}.
\]
$J$ need not be a compact set when $E$ is infinite (if not compact, it is analytic in the set theoretic sense of the term). One should also note that $E_A^{\infty,\om}$ is the set of all points $\rho\in E_A^{\infty}$ such that $A_{\om_{|\om|},\rho_1} = 1$. Alternatively, $E_A^{\infty,\om} = \sigma^{|\om|}([\om]_A)$. It is a good time to recall the definition of the strong open set condition which we will require of our systems.
\begin{definition}
    A CGDMS $\Phi$ satisfies the \textbf{strong open set condition (SOSC)} if $J\cap\text{Int}(X)\neq \emptyset$, that is, if 
    \[
        \bigcup_{v\in V} (J\cap \text{Int}(X_v)) \neq \emptyset.
    \]
\end{definition}

\noindent The coding map need not be injective either. For this reason, and for the sake of well-defined partitions later on when discussing the multifractal decomposition, we take
\[
    J_{>1} := \{x\in J\;|\; |\pi^{-1}(x)|>1\}.
\]
$J_{>1}$ is the set of points in the limit set with non-unique codings. The restriction $\displaystyle\pi\vert_{E_A^{\infty}\setminus \pi^{-1}(J_{>1})}$ is then bijective by construction.

There are also several types of potentials on the symbolic representation that we should take a moment to recall before developing our results.
\begin{definition}\label{def:summable}
    A potential  $\varphi: E_A^{\infty}\to \RR$ is \textbf{summable} if 
    \[
        \sum_{e\in E} \exp(\sup(\varphi\vert_{[e]}))<\infty.
    \]
    We say that for a CGDMS, $\Phi$, a family $F=\{f_e: X_{t(e)}\to \mathbb{R}\}_{e\in E}$ is summable if 
    \[
        \sum_{e\in E} \|\exp f_e\|_{X_{t(e)}}<\infty.
    \]
\end{definition}
\begin{definition}
    A potential $\varphi: E_A^{\infty}\to \RR$ is \textbf{acceptable} provided that it is uniformly continuous and 
    \[
        \osc(\varphi):= \osc(\varphi,\UU_E) = \sup_{e\in E}[\sup(\varphi\vert_{[e]})-\inf(\varphi\vert_{[e]})]<\infty.
    \]
\end{definition}

\noindent An important subclass of acceptable potentials are those which are H\"older continuous on cylinders.

\begin{definition}
    A potential $\varphi: E_A^{\infty}\to\RR$ is \textbf{H\"older continuous on cylinders} if there exists constants $\beta\geq 0$ and $c\geq 0$ for which 
    \[
        |\om \wedge \tau|\geq 1\implies |f(\om)-f(\tau)|\leq c[d(\om,\tau)]^{\beta},
    \]
    and we define $v_{\beta}(f)$ to be the least such $c$ with this property. We also say that for a CGDMS, $\Phi$, a family of potentials $F=\{f_e: X_{t(e)}\to \mathbb{R}\}_{e\in E}$ is \textbf{H\"older with exponent $\beta>0$} if 
    \[
        \nu_{\beta}(F):= \sup_{n\in\NN}\sup_{w\in E_A^n}\sup_{x,y\in X_{t(w)}}|f_{w_1}(\phi_{\sigma(w)}(x))-f_{w_1}(\phi_{\sigma(w)}(y))|e^{\beta n}<\infty.
    \]
\end{definition}
The finiteness of the inner two suprema implies that the diameters of $f_{w_1}\circ\phi_{\sigma(w)}(X_{t(w)})$ are finite for every finite word of length $n$ while the exterior supremum requires that these diameters go to $0$ exponentially fast as the length of the word goes to infinity. 
\begin{definition}\label{def:Variation}
    Let $f:E_A^{\infty}\to \RR$ be a real-valued potential. Define 
    \[
        \textit{var}_n(f) := \sup_{|\om\wedge\tau| = n}|f(\om)-f(\tau)|.
    \]
    We say $f$ is of \textbf{summable variations} if
    \[
        V(f):=\sum_{n=1}^{\infty}\textit{var}_n(f)<\infty. 
    \]
\end{definition}
\noindent There is a relationship between H\"older continuity on cylinders and summable variations. To see this, we recall a portion of the Bounded Variation Principle for Ergodic Sums that we will use later on.
\begin{lemma}[See \cite{urbanskiVolumeFinerThermodynamic2022} Lemma 17.2.3]\label{lem:BVErgodicSums} If a potential $f:E_A^{\infty}\to\RR$ is H\"older continuous on cylinders with exponent $\beta$, then for all $n\in \NN$, all $w\in E_A^{n}$ and all $\rho,\gamma\in E_A^{\infty,w}$, we have
\[
    |S_nf(w\rho)-S_nf(w\gamma)|\leq \frac{v_{\beta}(f)}{e^{\beta}-1}e^{-\beta|\rho\wedge\gamma|}.
\]
\end{lemma}
\noindent In particular, if a potential $f$ is H\"older continuous on cylinders with exponent $\beta$ taking $w = i\in \NN = E$ then each term in the sequence of differences $|f(w\rho_n)-f(w\gamma_n)|$ where $|\rho_n\wedge\gamma_n|=n\geq 0$ is bounded by a term in a geometric sequence independent of $i$. Taking supremums over $i$ and $|\rho_n\wedge\gamma_n|=n$ and summing then yields summable variations.

We also recall the construction of the amalgamated function of a family $F$. 

\begin{definition}
    For a family of potentials $F = \{f_e: X_{t(e)}\to \mathbb{R}\}_{e\in E}$ the \textbf{amalgamated function}  $\tilde{f}: E_A^{\infty}\to \mathbb{R}$ is given by 
    \[
        \tilde{f}(\om):=f_{\om_1}(\pi(\sigma(\om))).
    \]
\end{definition}

We note that the Lyapunov family $\log \Phi ':= \{\log|\phi_{e}'|\}_{e\in\NN}$ has amalgamated function $\Xip(\om) := \log|\phi_{\om_1}'(\pi(\sigma(\om)))|$, which is H\"older continuous on cylinders, and hence acceptable. For a CGDMS, if a family is H\"older with exponent $\beta$, then the amalgamated function is H\"older continuous on cylinders with exponent $\beta$, and a family is summable if and only if the amalgamated function is summable (See \cite[Lemma 19.8.4 and Lemma 19.8.6]{urbanskiVolumeFinerThermodynamic2022}). We note that for a CGDMS we have that 
\begin{equation}\label{eqn:EndBehaviorAlongAlphabet}
    \lim_{i\to\infty} \||\phi_i'|\|_{X_{t(i)}} = 0.
\end{equation}
For each family we also have the associated exponents defined by Birkhoff averages.
\begin{definition}
    Let $F$ be a family of potentials with amalgamated function $\tilde{f}$. For $\om\in E_A^{\infty}$ the \textbf{$F$-exponent} at $\om$, $\xi(\om)$ if it exists, is given by the Birkhoff average 
    \[
        \xi(\om) := \lim_{n\to\infty}\frac{1}{n}\sum_{j=0}^{n-1}\tilde{f}(\sigma^j(\om)).
    \]
    Further, for any shift invariant Borel probability measure $\mu$ we call the quantity 
    \[
        F(\mu):=\int_{E_A^{\infty}} \tilde{f}d\mu
    \]
    the characteristic $F$-exponent with respect to $\mu$. 
\end{definition}

A family of functions $F = \{f_e\}_{e\in E}$ is uniformly bounded below by $K\in\RR$ if and only if  $f_e\geq K$ for all $e\in E$. We say a family of functions is strictly positive if $F$ is uniformly bounded below and $K>0$. For a family $F$ the translation of the family by $a\in \RR$ is defined by $F-a := \{f_e-a:X\to \RR\}_{e\in E}$. Many of our results are stated for families of functions that are strictly positive. However, the following lemma allows us to preform the computations for the cases in which $F$ is uniformly bounded below by applying our theory to a translated family and then translating back to reflect the original family.

\begin{lemma}\label{lem:SpectrumIsTranslationOfFamilyInvariant}
    Let $K\in \RR$. If $F$ is a family of functions such that $F>K$, then we may replace $F$ with $F'=F-K+\varepsilon$ for some $\varepsilon>0$. That is, a point $\om$ has $F$-exponent $\xi_F$ if and only if $\om$ has $F'$ exponent $\xi_{F}-K+\varepsilon$
\end{lemma}
\begin{proof}
    The argument is quick. Let $\xi_F$ and $\xi_{F'}$ denote the $F$ and $F'$ exponent functions respectively. Note that if $\xi_F(\om)$ exists, then  
    \[
        \xi_F(\om) = \lim_{n\to\infty}\frac{1}{n}\sum_{j=0}^{n-1}\tilde{f}(\sigma^j(\om)) = \lim_{n\to\infty}\frac{1}{n}\sum_{j=0}^{n-1}[(\tilde{f}(\sigma^j(\om))+\varepsilon-K)-\varepsilon+K] = \xi_{F'}(\om)-\varepsilon+K.
    \]
\end{proof}
If $\mu$ is ergodic, then $\mu$-a.e. point has the same $F$-exponent. We also remark that the definition is different for Lyapunov exponents (with family $\log\Phi'$) but just up to a sign. In our framework however, the Lyapunov exponent may be studied by applying our results to the positive Lyapunov family $F= -\log\Phi'$
\begin{definition}\label{def:Lyapunov}
    Let $\Phi$ be a CGDMS, then the \textbf{characteristic Lyapunov exponent with respect to $\mu$} is given by 
    \[
        \chi(\mu) = -\int_{E_A^{\infty}}\Xip d\mu>0.
    \]
\end{definition}

\section{Topological Pressure}
Given CGDMS $\Phi = \{\phi_e\}_{e\in E}$, the corresponding pressure function of the family $t\log\Phi'+q F$ is given by the asymptotic growth rate of partition functions $Z_n(t,q)$:
\begin{equation}
    P(t,q) := P(t\log\Phi'+q F):= \lim_{n\to\infty}\frac{1}{n}\log\sum_{w\in E_A^n}||\exp(S_w(t\log\Phi'+q F))||_{X_{t(w)}} :=  \lim_{n\to\infty}\frac{1}{n}\log Z_n(t,q) \label{eqn:Pressure}
\end{equation}
where
\begin{equation}
    S_w(F) := \sum_{j=1}^{|w|} f_{w_j}\circ\phi_{\sigma^j(w)}
\end{equation}
and 
\begin{equation}
    \|\cdot \|_{X_t(\om)} := \| \cdot \vert_{X_t(\om)} \|_{\infty}.
\end{equation}
Note that the limit exists whenever $(t,q)$ produces a summable family as the sequence of sums is submultiplicative and so the convergence follows from Fekete's lemma \cite{feketeUeberVerteilungWurzeln1923}. This pressure agrees with the pressure of the amalgamated function $f_{t,q} =  t\Xip+q\tilde{f}$ of $t\log\Phi'+q F$ whenever $(t,q)$ produces a summable H\"older family of functions as per \cite[Proposition 19.8.9] {urbanskiVolumeFinerThermodynamic2022}. In this case, the amalgamated function is acceptable and, as we are concerned with finitely irreducible CGDMSs, by \cite[Proposition 17.2.8]{urbanskiVolumeFinerThermodynamic2022} we may write
\[
    P(t,q) = P(f_{t,q}) = \lim_{n\to\infty}\frac{1}{n}\log Z_n(f_{t,q}) = \lim_{n\to\infty}\frac{1}{n}\log\sum_{w\in E_A^n}\exp(\overline{S}_nf_{t,q}([w])).
\]

\noindent We concern ourselves with the following partition-like functions 
\begin{equation}
    \sum_{w\in E_A^n}\left\| \prod_{j=1}^{|w|}|\phi_{w_j}'|^t\circ \phi_{\sigma^j(w)}\right\|_{X_{t(w)}}\cdot \left\|\prod_{j=1}^{|w|} \exp(f_{w_j})^q\circ \phi_{\sigma^j(w)}\right\|_{X_{t(w)}} =: \Zt_n(t,q).
\end{equation}
\begin{proposition}\label{prop:JustifyManhattanRegion}
 $P(t,q)<\infty \iff \Zt_1(t,q)<\infty$. 
\end{proposition}
\begin{proof}
    According to \cite[Theorem 17.2.8]{urbanskiVolumeFinerThermodynamic2022} and \cite[Proposition 19.8.9]{urbanskiVolumeFinerThermodynamic2022} it is enough to show that $\Zt_n(t,q) \asymp Z_n(t,q)$. Let $K$ be the bounded distortion constant for $\Phi$ and let $\text{Dist}(F)$ be the distortion constant for $F$ given by exponentiation of the inequality in Lemma \ref{lem:BVErgodicSums} for $n=1$. Then, by the Cauchy-Schwarz inequality and the distortion properties for $F$ and $\Phi$, we have that
    \begin{equation}
        \max\{K^{-t},\text{Dist}(F)^{-q}\} \Zt_n(t,q) \leq Z_n(t,q) \leq \Zt_n(t,q). 
    \end{equation}
    Hence, $\Zt_1(t,q)< \infty \iff Z_1(t,q)< \infty \iff P(t,q)<\infty$.
\end{proof}

The geometric family of potentials, that is, the parameterized Lyapunov family $t\log\Phi'$, is of particular importance in the study of the geometry of $J$ for a given $\Phi$. For the choice of $F = -\log\Phi'$ we will denote the Lyapunov Pressure by $P_L(t,q) =: P(t-q,0)$, and sometimes we just write $P(t)=P_L(t,0)$ for this pressure function when $q=0$. The finiteness parameter $\theta$ and the Bowen parameter $h$ derived from the geometric family are of particular importance to us. They are defined as
\[
    \theta := \inf\{t \in\RR \;|\; P(t)<\infty\} \hspace{2cm} h:= \inf\{t\geq0\:|\; P(t)\leq 0\}.
\]
As noted in \cite[Chapter 19]{urbanskiVolumeFinerThermodynamic2022} $\theta\leq h$, $P(h)\leq 0$, and we have the following theorem known as Bowen's Formula.
\begin{theorem}[See \cite{urbanskiVolumeFinerThermodynamic2022} Theorem 19.6.4]
If $\Phi$ is a finitely irreducible CGDMS, then its Bowen parameter $h$ is such that 
\[
    h:=\inf\{t\geq 0\;|\; P(t)\leq 0\} = HD(J) = sup\{HD(J_F)\;|\; F\subseteq E, |F|<\infty\}\geq \theta.
\]
Moreover, if $P(t)=0$ then $t$ is the only zero of the pressure function and $t=h=HD(J)$.
\end{theorem}

\noindent The parameters allows us to recall the last set of important descriptors of CGDMS at play in our main theorem
\begin{definition}
       A CGDMS $\Phi$ is called \textbf{cofinitely regular} if $P(\theta) = \infty$.
\end{definition}
\begin{definition}
    A CGDMS $\Phi$ is called \textbf{regular} if there is some $t\geq 0$ such that $P(t)=0$. Equivalently, $P(h) = 0$. If a CGDMS is not regular it is said to be \textbf{irregular}. 
\end{definition}
Note that a cofinitely regular CGDMS is regular. This follows by \cite[Proposition 19.4.6 (c)]{urbanskiVolumeFinerThermodynamic2022} as the Lyapunov Pressure $P(t)$ is strictly decreasing to $-\infty$, continuous on $(\theta,\infty)$ and right-continuous at $\theta$. We will engage with cofinite regularity shortly in the discussion of the finiteness region of pressure in the next subsection. Before we do, we should also recall the variational principle for acceptable potentials and some special types of measures.
\begin{theorem} \label{thm:VP} [See \cite{urbanskiVolumeFinerThermodynamic2022} Theorem 17.3.4.] 
    If $\varphi:E_A^{\infty}\to \RR$ is an acceptable potential and $A$ is a finitely irreducible matrix, then 
    \begin{equation}\label{eq:VarPrinc}
        P(\varphi) = \sup_{\mu\in \MM}\left\{h_{\mu}(\sigma)+\int_{E_A^{\infty}}\varphi\; d\mu\right\},
    \end{equation}
    where the supremum is taken over $\MM$, the set of all $\sigma$-invariant probability measures $\mu$ on $E_A^{\infty}$ such that $\displaystyle\int_{E_A^{\infty}}\varphi\; d\mu>-\infty$. In fact, the supremum can be restricted to the subset of those measures that are ergodic.
\end{theorem}

\noindent Gibbs and equilibrium states are of note for the analysis completed in future sections, they are closely connected to the variational principal. 
\begin{definition}
    Let $\varphi:E_A^{\infty}\to\RR$ be an acceptable potential on a finitely irreducible shift space. An \textit{equilibrium state} $\mu_{\varphi}\in \mathcal{M}$ achieves the supremum in the variational principal \eqref{eq:VarPrinc} and  $\displaystyle\int_{E_A^{\infty}}\varphi d\mu_{\varphi}>-\infty$.
\end{definition}
\begin{definition}
    Let $\varphi:E_A^{\infty}\to\RR$ be an acceptable potential on a finitely irreducible shift space. A Borel probability measure $m_{\varphi}$ on $E_A^{\infty}$ is called a \textit{Gibbs state} for $\varphi$ if there exists a number $P\in \RR$ and a constant $C\geq 1$ such that for every $w\in E_A^{*}$ and every $\rho\in[w]$ the following string of inequalities holds:
    \begin{equation}\label{ineq:GSDef}
        C^{-1}\leq \frac{m_{\varphi}([w])}{\exp(S_{|w|}\varphi(\rho)-P|w|)}\leq C
    \end{equation}
\end{definition}
Potentials that are summable and H\"older continuous on cylinders have Gibbs states, and, under a mild integrability condition, Gibbs and equilibrium states coincide for finitely irreducible shifts.

\begin{theorem}\label{thm:Gibbs&EquilibriumStates} [See \cite{urbanskiVolumeFinerThermodynamic2022} Corollary 17.7.5.(c)] 
    Suppose that a potential $\varphi:E_A^{\infty}\to \RR$ is summable and H\"older continuous on cylinders and that the incidence matrix $A$ is finitely irreducible. Then the potential $\varphi$ has a unique $\sigma$-invariant Gibbs state $\mu_{\varphi}$, and this state is ergodic. If $\displaystyle \int_{E_A^{\infty}}\varphi\; d\mu_{\varphi}>-\infty$, then $\mu_{\varphi}$ is also the unique equilibrium state for $\varphi$.
\end{theorem}

\subsection{The Manhattan Region}
By analogy to hyperbolic surfaces (see \cite{pollicottWeilPeterssonMetrics2016}), we define the Manhattan region of pressure and its boundary, the Manhattan curve. The primary object of interest for us will be the Manhattan region and a particular subset of it which we call the multifractal region.
\begin{definition} For the family $t\log\Phi'+q F$ the set $D\subseteq \mathbb{R}^2$ given by 
\[
    D := \left\{(t,q)\in \RR^2\;\bigg|\; \sum_{i\in \NN} \|\exp(f_i)\|^q_{X_{t(i)}}\cdot \||\phi_i'|\|^t_{X_{t(i)}}<\infty\right\}
\]
is called \textbf{the Manhattan region of $P(t,q)$}. Its boundary, $\partial D$, is called \textbf{the Manhattan Curve of $P(t,q)$} Further, for a finitely irreducible CGDMS $\Phi$ with Bowen parameter $h$, we define 
\[
    D_0 := \{(t,q)\in D\; |\; 0\leq t\leq h\}
\]
as  \textbf{the multifractal region of $\Phi$ with respect to $F$}. 
\end{definition} 

\noindent  When $(t,q)\in D$ by Proposition \ref{prop:JustifyManhattanRegion} we have that $P(t,q)<\infty$. Note as well that on $\text{int}(D)$ the pressure $P(t,q)$ is real analytic (see \cite[Theorems 20.1.11, 20.2.3]{urbanskiVolumeFinerThermodynamic2022}). Under a cofinite regularity assumption on $\Phi$ and a comparability assumption between the families $F$ and $\log\Phi'$ we can show that $D$ is open.
\begin{lemma}\label{lem:Dopen}
If there exists $\alpha>0$ and $\beta,\gamma\in \RR$ such that 
\[
    -\alpha\log|\phi_i'|+\gamma \leq f_i \leq -\alpha\log|\phi_i'| +\beta
\]
and $\Phi$ is cofinitely regular, then $D\subset\mathbb{R}^2_{\text{std}}$ is open with respect to the standard topology.
\end{lemma}
\begin{proof}
    Suppose not, then there exsits $(t,q)\in \partial D\cap D$. Since $(t,q)\in D$ we have that 
    \[
       \Zt_1(t,q) = \sum_{i\in \NN} \|\exp(f_i)\|^q_{X_{t(i)}}\cdot \||\phi_i'|\|^t_{X_{t(i)}}<\infty.
    \]
    Note that for a fixed $q\in \mathbb{R}$ the function $\Zt_1(\cdot,q)$ is non-increasing, thus, since $(t,q)\in \partial D$, for every $\varepsilon>0$ given $(t-\varepsilon, q)\not\in D$. By the comparability assumption we have that
    \[
        \sum_{i=1}^{\infty}e^{\gamma q}\|\vert\phi_i'\vert|\|_{X_{t(i)}}^{-\alpha q+s} \leq \Zt_1(s,q) \leq \sum_{i=1}^{\infty}e^{\beta q}\|\vert\phi_i'\vert|\|_{X_{t(i)}}^{-\alpha q+s}.
    \]
    And so, $\Zt_1(s,q)\asymp \Zt_1(-\alpha q+s,0)$. Hence by definition of $\theta$ and from Proposition \ref{prop:JustifyManhattanRegion}, $(t,q)\in D$ if and only if $-\alpha q +t>\theta$. So there exists an $\varepsilon>0$ such that $-\alpha q +t -\varepsilon >\theta$. Let $a = \frac{\varepsilon}{\alpha v_2-v_1}$. Then $a>0$ and $\varepsilon = -av_1+a\alpha v_2$ and so, for this $a$,  $(t-av_1, q+av_2) = (t,q)+a\vec{v}\in D$, a contradiction. \hfill $\lightning$\newline
    Hence it must be the case that $D\cap \partial D = \emptyset$, and so $D\subset \mathbb{R}^2$ does not contain any of its boundary points. Thus, $D$ is open.
\end{proof}

\begin{remark}\label{rmk:LyaDistortion}
    The comparability condition can be framed in terms of families of functions as well. In this case, $F = -\alpha\log\Phi'+H$ where $H$ is a bounded H\"older family of functions. As evidenced by the Khintchine Decomposition given in \cite{fanKhintchineExponentsLyapunov2009}, these distortions of the Lyapunov family can yield different multifractal decompositions. 
\end{remark}

\begin{lemma}\label{lem:FBoundedDopen}
    If $F$ is bounded and $\Phi$ is cofinitely regular, then $D\subset \RR^2$ is open. In fact, it is the open half plane $\textup{Int}_{\RR^2}(H^+_{\theta}) = \{(t,q)\in\RR^2\:|\; t> \theta\}$.
\end{lemma}

\begin{proof}
     Let $M>F>K$ for some $M,K\in \RR$. Note that since $\Phi$ is cofinitely regular, we know that $(\theta,0)\notin D$ and $(t,0)\not\in D$ for all $t<\theta$ by definition of the finiteness parameter $\theta$. Then consider for any such $t<\theta$, we have that 
    \begin{equation}
        \Zt_1(t,q) > \exp(K)^q \Zt_1(t,0) = \infty. \; \forall q\in\RR
    \end{equation}
    Hence $(t,q)\not\in D$. Now consider the line $t = \theta$. Then 
    \begin{equation}
        \Zt_1(\theta,q) \geq \exp(K)^q \Zt_1(\theta,0) = \infty. \; \forall q\in\RR
    \end{equation}
    Thus, the line $t=\theta$ is not a subset of $D$, in fact, as $q$ was arbitrary, the line $t=\theta$ is the Manhattan curve of $P(t,q)$. Now, for all other $(t,q)$ where $t>\theta$,
    \begin{equation}
        \Zt_1(t,q) < \exp(M)^q \Zt_1(t,0) <  \exp(M)^q \Zt_1(\theta,0) < \infty.
    \end{equation}
    Hence, $D$ is the open half plane $\text{Int}_{\RR^2}(H^+_{\theta}) = \{(t,q)\in\RR^2\:|\; t> \theta\})$.
\end{proof}

When we drop cofinite regularity from the previous proposition, $D = H^+_{\theta}$. We will not analyze this case in this paper, but the previous proof requires only a few small changes, so it is a nice fact to include here. 
\begin{corollary}
    If $F$ is bounded and $\Phi$ is not cofinitely regular, then $D\subset \RR^2$ is the closed half plane $H^+_{\theta} = \{(t,q)\in\RR^2\:|\; t\geq \theta\}$.
\end{corollary}

\begin{proof}
    Let $M>F>K$ for some $M,K\in \RR$. Note that since $\Phi$ is not cofinitely regular, we know that $(\theta,0)\in D$ and $(t,0)\not\in D$ for all $t<\theta$ by definition of the finiteness parameter $\theta$. Then consider for any such $t<\theta$, we have that
    \begin{equation}
        \Zt_1(t,q) > \exp(K)^q \Zt_1(t,0) = \infty.\; \forall q\in\RR
    \end{equation}
    Hence $(t,q)\not\in D$. Now consider the line $t = \theta$. Then 
    \begin{equation}
        \Zt_1(\theta,q) < \exp(M)^q \Zt_1(\theta,0) < \infty. \; \forall q\in\RR
    \end{equation}
    Thus, the line $t=\theta$ is a subset of $D$ the rest of the proof is identical to the previous proposition.
\end{proof}

By taking the contra-positive of the previous lemmas, and noting that $D$ is always a proper subset of $\RR^2$ in the infinite alphabet case and $D = \RR^2$ in the finite alphabet case, we yield the following two corollaries.

\begin{corollary} 
If $D$ is open, then $\Phi$ is cofinitely regular or $F$ is not bounded.
\end{corollary}
\begin{proof}
    If not, then $D$ is open and $F$ is bounded and $\Phi$ is not cofinitely regular, so $D = H_{\theta}^+$ and thus is clopen. However $\RR^2$ is connected so $D$ cannot be a proper clopen subset of $\RR^2$. So either $D = \RR^2$ which is open and $\Phi$ would be cofinitely regular with $\theta = -\infty$ (the case of finite alphabet systems) or, we have a contradiction and $D\subseteq \RR^2$ is open and $F$ is not bounded or $\Phi$ is cofintely regular as desired. 
\end{proof}

\begin{corollary}
If $D$ is neither open nor closed, then $F$ is not bounded.
\end{corollary}

\subsection{Derivatives of Pressure}

\noindent Under a mild integrability assumption, we can describe the first derivatives of $P(t,q)$ as well. 

\begin{lemma}[ND]\label{lem:DerivativesOfPressure}
    Suppose $(t,q)\in D$ and that $\displaystyle \int_{E_A^{\infty}}(|\tilde{f}|+|\Xip|)\;d\tmu_{t,q}<\infty$ where $\tmu_{t,q}:= \mu_{f_{t,q}}$ from Theorem \ref{thm:Gibbs&EquilibriumStates}, then for $F$ strictly positive,
    \begin{equation}
        \frac{\partial P}{\partial t} = \int_{E_A^{\infty}}\Xip\;d\tmu_{t,q}<0 \qquad 
        \text{ and } \qquad \frac{\partial P}{\partial q} = \int_{E_A^{\infty}}\tilde{f}\;d\tmu_{t,q}>0.
    \end{equation}
\end{lemma}
\begin{proof}
    The argument presented is similar to \cite[Theorem 16.4.10]{urbanskiVolumeFinerThermodynamic2022}. Since $(t,q)\in D$, the family $t\log\Phi'+q F$ is summable and H\"older, so, the amalgamated function $f_{t,q}$ is summable and holder continuous on cylinders. Hence, $\tmu_{t,q}$ exists. Then, for $t'> t$, by the variational principle (Theorem \ref{thm:VP}),
    \begin{align*}
        P(t',q)&\geq h(\tmu_{t,q})+\int_{E_A^{\infty}}(t'\Xip+q\tilde{f})\;d\tmu_{t,q}\\
        &= h(\tmu_{t,q})+\int_{E_A^{\infty}}f_{t,q}\;d\tmu+(t'-t)\int_{E_A^{\infty}}\Xip\;d\tmu_{t,q}\\
        &=P(t,q)+(t'-t)\int_{E_A^{\infty}}\Xip\;d\tmu_{t,q}.
    \end{align*}
    And so,
    \begin{equation}
        \frac{\partial P}{\partial t}(t,q) = \lim_{t'\to t^{+}}\frac{P(t',q)-P(t,q)}{t'-t}\geq \int_{E_A^{\infty}}\Xip\;d\tmu_{t,q}.
    \end{equation}
    Similarly, when $t'<t$ we have that 
    \begin{equation}
        \frac{\partial P}{\partial t}(t,q) = \lim_{t'\to t^{-}}\frac{P(t',q)-P(t,q)}{t'-t}\leq \int_{E_A^{\infty}}\Xip\;d\tmu_{t,q}<0.
    \end{equation}
    The argument for $\frac{\partial P}{\partial q}$ is identical. However, the resulting $q$-derivative is positive since $\tilde{f}\geq K>0$. Thus, the derivatives are as stated.
\end{proof}
We will also want to avoid the case in which our amalgamated functions are cohomologous to a constant. This, as we will see, will allow us to retrieve some information about the second derivatives of $P(t,q)$. Luckily, we have the following:

\begin{lemma}[ND]\label{lem:ND-Cohom} Let $\Phi$ be a cofinitely regular countably infinite CGDMS that is finitely irreducible and cofinitely regular and $F$ be bounded below. Then, for $(t,q)\in D$ the amalgamated function $f_{t,q} = t\Xip+q\tilde{f}$ is not cohomologous to a constant.
\end{lemma}
\begin{proof}
    Suppose first that $t\neq 0$ and $f_{t,q}$ is cohomologous to a constant. Then there is a bounded function $u$ for which there exists $C\in \RR$ such that
    \begin{equation}
        t\Xip+q\tilde{f} = u-u\circ\sigma+C.
    \end{equation}
    However, taking a Birkhoff average one obtains
    \begin{equation}
        \lim_{n\to\infty}\frac{1}{n}\sum_{j=0}^{n-1}(t\Xip+q\tilde{f})(\sigma^j(\om)) = C \qquad \forall \om\in E_A^{\infty}.
    \end{equation}
    
    Now suppose $t>0$. since $\Phi$ is finitely irreducible, $A$ is a finitely irreducible matrix, so there exists a finite set of finite words $I\subset E_A^*$ which witnesses the finite irreducibility of $E_A^{*}$. That is, for each $e_1,e_2\in E$ there exists a $\tau(e_1,e_2)\in I$ such that $e_1\tau(e_1,e_2)e_2$ is $A$-admissible. Set
    \begin{equation}
    E_{w} := \{k\in E\;|\; (kw)^{\infty}\in E_A^{*}\}\;\; \forall w\in I
    \end{equation} 
    By the pigeonhole principle, we may fix a $\tau\in I$ such that $|E_{\tau}|=\infty$. For this $\tau$ we have that 
    \begin{equation}
        \lim_{n\to\infty}\frac{1}{n}\sum_{j=0}^{n-1}(t\Xip+q\tilde{f})(\sigma^j((k\tau)^{\infty})) = C_k = C \qquad \forall k\in E_{\tau}.
    \end{equation}
    Since the point is periodic, the Birkhoff average is just the arithmetic mean of the potential evaluated at elements of the finite orbit so the equation simplifies to
    \begin{align}
        t\Xip(&(k\tau)^{\infty})+t\Xip((\tau k)^{\infty})+\dots +t\Xip((\tau_{|\tau|}k\tau_1\dots \tau_{|\tau|-1})^{\infty})\notag\\&+ q\tilde{f}((k\tau)^{\infty})+q\tilde{f}((\tau k)^{\infty})+\dots +q\tilde{f}((\tau_{|\tau|}k\tau_1\dots \tau_{|\tau|-1})^{\infty}) = (|\tau|+1)C_k,\; \forall k\in E_{\tau}.
    \end{align}
    Note that all terms except for $t\Xip((k\tau)^{\infty})$ and $ q\tilde{f}((k\tau)^{\infty})$ may be bounded since the amalgamated functions $\Xip$ and $\tilde{f}$ are H\"older continuous on cylinders and thus acceptable. In fact, the bound on these terms may be chosen such that it is only dependent on $\tau$. Now by (\ref{eqn:EndBehaviorAlongAlphabet}) and since $t>0$ we know that $\displaystyle\lim_{k\to\infty} t\Xip((k\tau)^{\infty}) = -\infty$, so  we must have that $\displaystyle\lim_{k\to\infty} q\tilde{f}((k\tau)^{\infty}) = \infty$. 
    And so,  
    \begin{equation}
        \lim_{k\to\infty} -\frac{\Xip((k\tau)^{\infty})}{\tilde{f}((k\tau)^{\infty})} = \frac{q}{t}.
    \end{equation}
    By passing to subsequences if necessary, we may suppose that the limit approaches its value from above or below. In the case that the limit approaches from below, we consider $q'>q$, then for all $k$ large enough, there exists an $s(k)\leq \frac{q}{q'}<1$ such that
    \begin{align}
        &-\frac{t\log|\phi_k'(\pi((\tau k)^{\infty}))|}{q' \tilde{f}((k\tau)^{\infty})}\leq s(k)<1\notag\\
        &\implies -t\log|\phi_k'(\pi((\tau k)^{\infty}))|-s(k)q'f((k\tau)^{\infty})\leq0\notag\\
        &\implies t\log |\phi_k'(\pi((\tau k)^{\infty}))|+s(k)q'f((k\tau)^{\infty})\geq0\notag\\
        &\implies \| |\phi_k'| \|^t_{X_{t(k)}}\cdot \|\exp(f_k)\|^{s(k)q'}_{X_{t(k)}} > 1
    \end{align}
    Now noting that $s(k)q'\leq q$ and that $\|\exp(f_k)\|_{X_{t(k)}} > 1$ for $k$ large enough, we obtain
    \begin{align}
         \| |\phi_k'| \|^t_{X_{t(k)}}\cdot \|\exp(f_k)\|^{q}_{X_{t(k)}}\geq1 \implies (t,q)\not\in D. \qquad \lightning
    \end{align}
    Hence it must be the case that the limit approaches its value from above. However, here we consider $t'<t$. Then for $k$ large enough there exists an $s(k)\geq \frac{t}{t'}>1$ such that
    \begin{align}
        &-\frac{t'\log|\phi_k'(\pi((\tau k)^{\infty}))|}{q \tilde{f}((k\tau)^{\infty})}\leq \frac{1}{s(k)}<1\notag\\
        &\implies -t'\log|\phi_k'(\pi((\tau k)^{\infty}))|-\frac{1}{s(k)}qf((k\tau)^{\infty})\leq0\notag\\
        &\implies -t's(k)\log|\phi_k'(\pi((\tau k)^{\infty}))|-qf((k\tau)^{\infty})\leq0\notag\\
        &\implies t's(k)\log |\phi_k'(\pi((\tau k)^{\infty}))|+qf((k\tau)^{\infty})\geq0\notag\\
        &\implies \| |\phi_k'| \|^{t's(k)}_{X_{t(k)}}\cdot \|\exp(f_k)\|^{q}_{X_{t(k)}} > 1
    \end{align}
    Now noting that $s(k)t'\geq t$ and that $\| |\phi_k'| \|_{X_{t(k)}} < 1$ for $k$ large enough, we obtain
    \begin{align}
         \| |\phi_k'| \|^t_{X_{t(k)}}\cdot \|\exp(f_k)\|^{q}_{X_{t(k)}}\geq1 \implies (t,q)\not\in D. \qquad \lightning
    \end{align}
    
    The case of $t<0$ is similar. First note that $q<0$ is required in this case. Now, assuming that the limit is approaching from above, the reader should consider $q'<q<0$ and for large enough $k$ the existence of an $s(k)$ such that $\frac{q}{q'} \leq s(k) < 1$.
    \begin{equation}
        -\frac{t\log|\phi_k'(\pi((\tau k)^{\infty}))|}{q' \tilde{f}((k\tau)^{\infty})}\geq s(k)
    \end{equation}
    And, when approaching from below, the reader should consider $0>t'>t$ and the existence of an $s(k)\geq \frac{t}{t'}>1$ such that 
    \begin{equation}
        -\frac{t'\log|\phi_k'(\pi((\tau k)^{\infty}))|}{q \tilde{f}((k\tau)^{\infty})}\geq \frac{1}{s(k)}
    \end{equation}
    Both computations give contradictions, and hence $f_{t,q}$ is not cohomologous to a constant when $t\neq 0$.
    
    Now if $t=0$, $qF$ cannot be bounded as in that case $(0,q)\not\in D$. So $F$ is unbounded in this case. Without loss of generality we may assume that $F$ is strictly positive per Lemma \ref{lem:SpectrumIsTranslationOfFamilyInvariant}. If $(0,q)\in D$ then we must also have $q<0$. So using finite irreducibility of $\Phi$, There exists two points of finite period $w^{\infty}$ and $u^{\infty}$ such that $|w|,|u|\leq|I|+1$. Further as $F$ is unbounded, we can choose $w$ and $u$ such that
    \begin{equation}
        \lim_{n\to\infty}\frac{1}{n}\sum_{j=0}^{n-1}(q\tilde{f})(\sigma^j(w)) = C_w < 0
    \end{equation}
    and 
    \begin{equation}
        q\tilde{f}(u^{\infty}) < C_w\cdot(\vert I\vert+1) \implies \lim_{n\to\infty}\frac{1}{n}\sum_{j=0}^{n-1}(q\tilde{f})(\sigma^j(u^{\infty})) < C_w.
    \end{equation}
    Where the last inequality follows not only from the choice of $u$, but also from each term in the average being negative. One may take $w$ and $u$ to have first and last symbol in $E_{\tau}$ to achieve such a setting. Hence $q\tilde{f}$ is not cohomologous to a constant.\hfill $\lightning$ 
    
    Thus, in all cases $f_{t,q} = t\Xip+q\tilde{f}$ is not cohomologous to a constant and we are done. 
\end{proof}

\begin{remark}
    A truncated form of the above proof also applies to the case where $F$ is bounded as we arrive at $\displaystyle\lim_{k\to\infty} q\tilde{f}((k\tau)^{\infty}) = \infty$, a contradiction $\lightning$.
\end{remark}
With the cohomology established, we recall the following result:
\begin{theorem}  [See \cite{fanKhintchineExponentsLyapunov2009} Theorem 4.2] 
    Let $\Psi$ and $F$ be real-valued H\"older family of functions. If $t\tilde{\psi}+q\tilde{f}$ is not cohomologous to a constant function, then $P(t,q)$ is strictly convex and 
    \[
        H(t,q):= \left[\begin{matrix} \frac{\partial^2 P}{\partial t^2} & \frac{\partial^2 P}{\partial t\partial q} \\[8pt] \frac{\partial^2 P}{\partial t\partial q} & \frac{\partial^2 P}{\partial q^2}  \end{matrix}\right]
    \]
    is positive semi-definite.
\end{theorem}
\noindent The positive semi-definite property follows from the strict convexity of $P$.

\begin{lemma}[ND]\label{lem:PDBoundaryLimits} If $D$ is open we have the following
\begin{enumerate}
    \item   For any $(t_0,q_0)\in\partial D$ 
    \begin{equation}\label{eqn:BoundaryDPressure}
        \lim_{(t,q)\to(t_0,q_0)}P(t,q) =\infty. 
    \end{equation}
    \item For fixed $t\in\mathbb{R}$ and $F$ strictly positive
    \begin{equation}\label{eqn:BoundaryDPressureDerivative}
        \lim_{(t,q)\to (t,q_0)\in \partial D} \frac{\partial P}{\partial q}(t,q)= \infty.
    \end{equation}
\end{enumerate}
\end{lemma}
\begin{proof}
   Let $D\ni(t,q)\to(t_0,q_0)\in \partial D$ for some given $(t_0,q_0)$. Since $(t,q)\in D$, $P(t,q)<\infty$ and 
   \begin{equation}
        P(t,q) = P(f_{t,q}).
   \end{equation}
   From \cite[Lemma 17.2.5]{urbanskiVolumeFinerThermodynamic2022}, since $f_{t,q}$ is H\"older continuous on cylinders $\{Z_n(f_{t,q})\}_{n\in\NN}$ is submultiplicative and boundedly supermultiplicative, $Z_1(f_{t,q})^n\geq Z_n(f_{t,q})$ and there exists a $0<Q<\infty$ such that $Q^{-n}Z_1(f_{t,q})^n\leq Z_n(f_{t,q})$. Further, by \cite[Theorem 17.2.8]{urbanskiVolumeFinerThermodynamic2022}, as $f_{t,q}$ is acceptable, $P(f_{t,q}) = \displaystyle\inf_{n\in\NN}\frac{1}{n}\log Z_n(f_{t,q})$. Hence,
   \begin{equation}
        -\log Q + \log Z_1(f_{t,q}) \leq P(t,q)\leq \log Z_1(f_{t,q}).
   \end{equation}
   Since the family $t\log\Phi'+qF$ is H\"older for all $(t,q)\in \RR^2$, $v_\beta(t\log\Phi'+qF)<\infty$. By  \cite[Proposition 19.8.9]{urbanskiVolumeFinerThermodynamic2022} with $V_{\beta}(t,q) := \exp\left(\frac{v_{\beta}(t\log\Phi'+qF)}{1-e^{-\beta}}\right)$,
   \begin{equation}
       V_{\beta}(t,q)^{-1}\|\exp(S_w(t\log\Phi'+qF))\|_{X_{t(w)}} \leq \exp(\overline{S}_{|w|}f_{t,q}([w])) \implies \nu_{\beta}(t,q)^{-1}Z_1(t,q)\leq Z_1(f_{t,q}).
   \end{equation}
    Hence,
   \begin{equation}\label{ineq:ZSqueezeP}
        -\log Q - \log V_{\beta}(t,q) + \log Z_1(t,q) \leq P(t,q)\leq \log Z_1(f_{t,q}).
   \end{equation}
    Note that 
    \begin{equation}
        v_\beta(t\log\Phi'+qF)\leq |t|v_\beta(\log\Phi')+|q|v_\beta(F).
    \end{equation} 
    Thus, we may take a bounding diamond about a tail of the path from $(t,q)$ to the $(t_0,q_0)$ in question in the limit and bound $V_\beta(t,q)$ for $(t,q)$ sufficiently close to $(t_0,q_0)$ by using the corner of the diamond with the maximum distance from the origin. Thus, by Proposition \ref{prop:JustifyManhattanRegion}, $\Zt_1(t,q)\asymp Z_1(t,q)$, so it is enough to show that $\Zt_1(t,q)\to\infty$ as $(t,q)\to(t_0,q_0)\in \partial D$. Since $D$ is open, $\partial D\cap D = \emptyset$, and so we have $\Zt_1(t_0,q_0) = \infty$. Suppose now, by way of contradiction, that the limit is finite or does not exist. In either case there is a path along which the limit is finite.  Then, by Fatou's lemma
   \begin{align}
        \infty > M = \lim_{(t,q)\to(t_0,q_0)} \Zt_1(t, q) &\geq \sum_{i=1}^{\infty}\lim_{(t,q)\to(t_0,q_0)}\|\exp(f_i)\|^{q}\||\phi'|\|^t\notag\\
        &=\sum_{i=1}^{\infty}\|\exp(f_i)\|^{q_0}\||\phi'|\|^{t_0} = \infty \quad \lightning 
   \end{align}
    \noindent And so, by this contradiction it must be the case that (\ref{eqn:BoundaryDPressure}) holds. 

    Now for (\ref{eqn:BoundaryDPressureDerivative}), by convexity of $P(t,q)$ for $q>q'$ and by (\ref{eqn:BoundaryDPressure}) we have
    \begin{align}
        \frac{\partial P}{\partial q}(t,q)&\geq \frac{P(t,q)-P(t,q')}{q-q'} \Rightarrow\notag\\  &\lim_{(t,q)\to(t,q_0)\in \partial D}  \frac{\partial P}{\partial q}(t,q) \geq \lim_{(t,q)\to(t,q_0)\in \partial D} \frac{P(t,q)-P(t,q')}{q-q'} = \infty
    \end{align}
    as desired.
    \end{proof}

\subsection{Conformality of Measures}
Under the strong open set condition we can confirm the $t\log \Phi'+qF$-conformality of the measure $m_{t,q}= \tilde{m}_{t,q}\circ\pi^{-1}$ induced by the conformal measures $\tilde{m}_{t,q}$ on $E_A^{\infty}$ from our main theorem. We first recall the definition of a conformal measure for a CGDMS, it is provided for completeness.

\begin{definition}
    Let $\Phi=\{\phi_e\}_{e\in E}$ be a CGDMS and $F = \{f_e\}_{e\in E}$ be a H\"older family of functions. A Borel probability measure $m$ on $X$ is said to be \textbf{$F$-conformal} provided it is supported on the limit set $J$ and 
    \begin{enumerate}
        \item For every $e\in E$ and for every Borel set $B\subseteq \pi(E_A^{\infty})$
        \[
            m(\phi_e(B)) = \int_B \exp(f_e-P(F))dm.
        \]
        \item For all $d\neq e\in E$
        \[
            m(\phi_d(X_{t(d)})\cap\phi_e(X_{t(e)})) = 0.
        \]
    \end{enumerate}
\end{definition}

\noindent The following theorems confirm the conformality of the measure $m_{t,q}$ on $J$.
\begin{theorem}[\cite{urbanskiVolumeFinerThermodynamic2022} Theorem 19.7.2.(c)]
Let $\Phi=\{\phi_e\}_{e\in E}$ be a CGDMS satisfying the SOSC then for any ergodic $\sigma$-invariant Borel probability measure $\mu$ on $E_A^{\infty}$ such  that $\text{supp}(\mu) = E_A^{\infty}$ and for incomparable words $\om,\tau\in E_A^{*}$ we have that: 
\begin{equation}\label{eqn:TheImportantConformal-LikeProperty}
    \mu\circ\pi^{-1}(\phi_{\om}(X_{t(\om)})\cap\phi_{\tau}(X_{t(\tau)}))=0.
\end{equation}
\end{theorem}
\begin{theorem}[\cite{urbanskiVolumeFinerThermodynamic2022} Theorem 19.8.14] Let $\Phi$ be a finitely irreducible CGDMS for which (\ref{eqn:TheImportantConformal-LikeProperty}) holds for any ergodic $\sigma$-invariant Borel probability measure $\mu$ on $E_A^{\infty}$ such that $\text{supp}(\mu) = E_A^{\infty}$. Then for any summable H\"older family of functions $F = \{f_e\}_{e\in E}$, the CGDMS $\Phi$ has a unique $F$-conformal measure $m_F$. Moreover
\[
    m_{F} = m_{f}\circ\pi^{-1},
\]
where $f$ is the amalgamated function induced by $F$ and $m_{f}$ is the eigenmeasure of the dual transfer operator $\LL^*_f$.
\end{theorem}

\section{Tightness and Zero-Temperature Style Limits}
There are a few issues that arise in the argument of the Lyapunov spectrum of the Gauss map in \cite[Proposition 6.1]{fanKhintchineExponentsLyapunov2009}. The errors concern the compactness of $\mathcal{M}$ and the upper semi-continuity of metric entropy, properties that are well known in the case of finite alphabet shifts of finite type (SFTs) and other expansive continuous self-maps of non-empty compact metrizable spaces. \cite[Chapter 5, Theorem 12.2.5]{urbanskiVolumeErgodicTheory2021} However, these properties require more care in the case of the continued fraction Gauss map and other systems that may be represented by countably infinite alphabet shifts. We will work under the assumptions that we have a finitely irreducible symbolic space and $F$ is bounded below. Note that the comparability condition between the Lyapunov and $F$ families from Lemma \ref{lem:Dopen} implies that $F$ is bounded below. In light of Lemma \ref{lem:SpectrumIsTranslationOfFamilyInvariant}, without loss of generality we will work as if $F$ is strictly positive. What we argue here will repair the non-fatal errors made in the derivation of the Lyapunov spectrum. (One may also adapt and consult \cite{jenkinsonZeroTemperatureLimits2005} and \cite{freireEquilibriumStatesZero2018} to repair these issues in the context of the Gauss map directly.) These arguments are also pertinent to our derivation of the $F$-exponent spectrum. Since our family $F$ is strictly positive, and since $\frac{\partial P}{\partial q}(t,\cdot) = \int_{E_A^{\infty}} \tilde{f}\; d\tmu_{t,\cdot}$ is continuous and strictly increasing, we know that 
\begin{equation}
    A = \left\{\frac{\partial P}{\partial q}(t,\cdot)\;\Big\vert\; (t,q)\in D\right\}
\end{equation}
is the continuous image of an interval. So, $A$ is an interval with 
\begin{equation}
    \inf(A) = \inf_{q}\left\{\int_{E_A^{\infty}}\tilde{f}\;d\tmu_{t,q}\right\} =: \xi_{\min}(t).
\end{equation}
The goal of this section is to show that $\xi_{\min}(t)$ is independent of $t$, that there exists $\mu\in\MM$ for which this $\xi_{\min}:= \xi_{\min}(t)$ is achieved, and that it results from a zero-temperature style limit of Gibbs States. For background on zero-temperature limits see \cite{jenkinsonZeroTemperatureLimits2005}.
\subsection{Tightness and Prohorov's Theorem}
We begin with one of our main results, the tightness result for the relevant Gibbs states.
\begin{customprop}{\ref{MainResult2}} [c.f. \cite{jenkinsonZeroTemperatureLimits2005} Lemma 2]Fix $0\leq t\leq h$ and let $q\in (-\infty,q_1)\subsetneq(-\infty,-1)$ such that $(0,q_1)\in D$. Then the families of Gibbs states $\{\{\tmu_{t,q}\}_{q\leq q_1}\}_{0\leq t\leq h}$ are tight, i.e. for all $\varepsilon>0$ there exists $K\subseteq E_A^{\infty}$ compact such that $\tmu_{t,q}(K)>1-\varepsilon$ for all $q<q_1$.
\end{customprop}
\begin{proof}
    Note that since $(0,q_1)\in D$ and since $P(t,q)$ is decreasing in $t$, we have that $(t,q_1)\in D$ for all $0\leq t\leq h$. Further, $(t,q)\in D$ for all $q\leq q_1$ since $P(t,q)$ is increasing in $q$. Next we fix a $t\in[0, h]$. Note that since $f_{t,q}$ is H\"older continuous on cylinders, it is also of summable variations by Lemma \ref{lem:BVErgodicSums}. Now let $\varepsilon>0$ and note that the set
    \begin{equation}\label{eqn:TheCompactK}
        K = \{x\in E_A^{\infty}\;|\; 1\leq x_k\leq a_k,\; \forall k\in\NN\}
    \end{equation}
    is compact for any $\{a_k\}_{k\in\NN}\in\NN^\infty$ by Tychonoff's Theorem. Setting $K$ aside for a moment we begin our analysis with the Gibbs property which we recall here in line with \cite[Theorem 2.2.7]{mauldinGraphDirectedMarkov2003} and \cite{freireEquilibriumStatesZero2018}. 
    \begin{align}\label{ineq:tqGibbs1}
        e^{-4V(f_{t,q})}\leq &\frac{\tmu_{t,q}([w])}{\exp(S_nf_{t,q}-|w|P(t,q))}\leq e^{4V(f_{t,q})}\nonumber\\
        &\Rightarrow \tmu_{t,q}([a]) \leq e^{4V(f_{t,q})}\exp(\sup\{f_{t,q}\vert_{[a]}\} -P(t,q)) \hspace{.25cm} \forall a\in E=\NN.
    \end{align}
    
    Now let $m\in\MM$ such that $I=\int_{E_A^{\infty}}f_{t,q_1}dm<\infty$. Such an $m$ exists since $E_A^{\infty}$ is finitely irreducible, and so, there exists $\tau\in E_A^*$ such that $\tau^{\infty}\in E_A^{\infty}$ and we may take $m\in \MM$ supported on $\OO(\tau^{\infty})$, the orbit of 
    $\tau^{\infty}$. By the variational principle (Theorem \ref{thm:VP}) and since $0>q_1>q$, $\frac{q}{q_1}>1$,
    \begin{align}\label{eqn:P1}
        P(t,q)-\frac{q}{q_1}\int_{E_A^{\infty}} f_{t,q_1}\;dm &= P\left(f_{t,q}-\frac{q}{q_1}\int_{E_A^{\infty}} f_{t,q_1}\;dm \right) \nonumber \\
        &\geq \int_{E_A^{\infty}}\left(f_{t,q}-\frac{q}{q_1}\int_{E_A^{\infty}} f_{t,q_1}\;dm\right)dm+h(m)\nonumber\\
        & = t\left(1-\frac{q}{q_1}\right)\int_{E_A^{\infty}}\Xip\;dm +h(m)\geq 0.
    \end{align}
    Hence applying (\ref{eqn:P1}) to the Gibbs property inequality of interest (\ref{ineq:tqGibbs1}) grants
    \begin{align}\label{ineq:ME1}
        \tmu_{t,q}([a])&\leq e^{4V(f_{t,q})}\exp\left(\sup\left\{\left(f_{t,q}-\frac{q}{q_1}I\right)\Bigg\vert_{[a]}\right\}\right) e^{-P(f_{t,q}-\frac{q}{q_1}I)}\nonumber\\
        &\leq e^{4V(f_{t,q})}\exp\left(\sup\left\{\left(f_{t,q}-\frac{q}{q_1}I\right)\Bigg\vert_{[a]}\right\}\right)\nonumber\\
        &= \exp\left(4V(f_{t,q})-\frac{q}{q_1}\int_{E_A^{\infty}} f_{t,q_1}\;dm+\sup\{f_{t,q}\vert_{[a]}\}\right)\nonumber\\
        &=\exp\left(\frac{q}{q_1}\left(4V\left(f_{\frac{q_1}{q}t,q_1}\right)-I+\sup\left\{f_{\frac{q_1}{q}t,q_1}\vert_{[a]}\right\}\right)\right)
    \end{align}
    since $kf_{t,q} = f_{kt,kq}$ for any $k\in \RR$. Note that $(q_1<-1)\Rightarrow \left(\frac{q_1}{q}<1<|q_1|\right)$. So by Definition (\ref{def:Variation}) and the triangle inequality, we have that 
    \begin{equation}\label{ineq:variation1}
        4V\left(f_{\frac{q_1}{q}t,q_1}\right)\leq \frac{4q_1}{q}tV(\Xip)+4|q_1|V(\tilde{f}) \leq 4|q_1||t|V(\Xip)+4|q_1|V(\tilde{f}).
    \end{equation}
    Combining (\ref{ineq:ME1}) and (\ref{ineq:variation1}) gives
    \begin{align}\label{ineq:ME2}
        \tmu_{t,q}([a])&\leq \exp\left(\frac{q}{q_1}\left(4|q_1|tV(\Xip)+4|q_1|V(\tilde{f})-I+\sup\left\{f_{\frac{q_1}{q}t,q_1}\vert_{[a]}\right\}\right)\right)\nonumber\\
        &\leq \exp\left(\frac{q}{q_1}\left(4|q_1|tV(\Xip)+4|q_1|V(\tilde{f})-I+\sup\left\{f_{0,q_1}\vert_{[a]}\right\}\right)\right).
    \end{align}
    Further, $f_{0,q_1}$ is summable, so by Definition \ref{def:summable} and since $q<q_1$, it must be the case that $\sup\{f_{t,q}\vert_{[a]}\}\leq \sup\{f_{0,q_1}\vert_{[a]}\} \to -\infty$ as $a\to\infty$. So, there exists an $N\in \NN$ independent of $q$ such that for all $a\geq N$ we have that 
    \begin{equation}
        \frac{q}{q_1}\left(4|q_1|tV(\Xip)+4|q_1|V(\tilde{f})-I+\sup\left\{f_{0,q_1}\vert_{[a]}\right\}\right)<0.
    \end{equation}
    Continuing with (\ref{ineq:ME2}) by taking $a\geq N$ and again noting that $\frac{q}{q_1}\geq1$,
    \begin{align}\label{ineq:ME3}
        \tmu_{t,q}([a])\leq \exp\left(4|q_1|tV(\Xip)+4|q_1|V(\tilde{f})-I+\sup\left\{f_{0,q_1}\vert_{[a]}\right\}\right),
    \end{align}
    and, as $(0,q_1)\in D$, we may find for each $k\in\NN$ an $a_k\geq N$ such that 
    \begin{align}\label{ineq:Sum1}
        \sum_{i = a_k+1}^{\infty} e^{\sup\{f_{0,q_1}\vert_{[i]}\}}<\frac{\varepsilon}{2^k}e^{\left(-4|q_1||t|V(\Xip)-4|q_1|V(\tilde{f})+I\right)}.
    \end{align}
    We now consider the compact $K$ from \eqref{eqn:TheCompactK} defined by the $a_k$'s. Define $\pi_k:E_A^{\infty}\to E$ by $\pi_k(\om) :=\om_k$. Then, 
    \begin{equation}
        \pi_k^{-1}(i):=\{\om\in E_A^{\infty}\;|\; \om_k = i\}
    \end{equation}
    And so,
    \begin{align} \label{ineq:ME4}
        \tmu_{t,q}(K) &= \tmu_{t,q}\left(E_A^{\infty}\setminus \bigcup_{k=1}^{\infty}\{\om\in E_A^{\infty}\;|\; \om_k > a_k\}\right) \nonumber \\
        &\geq 1-\sum_{k=1}^{\infty} \tmu_{t,q}(\{\om\in E_A^{\infty}\;|\; \om_k>a_k\}) \nonumber\\
        & = 1-\sum_{k=1}^{\infty}\sum_{i = a_k+1}^{\infty}\tmu_{t,q}(\pi_k^{-1}(i)) \nonumber \\
         & = 1-\sum_{k=1}^{\infty}\sum_{i = a_k+1}^{\infty}\tmu_{t,q}([i]).
    \end{align}
    Combining (\ref{ineq:ME4}), (\ref{ineq:Sum1}), and (\ref{ineq:ME3}), we obtain integers $a_k$ such that 
    \begin{equation}
        \sum_{i=a_k+1}^{\infty} \tmu_{t,q}([i])< \frac{\varepsilon}{2^k} \text{ for all } k\in \NN, q< q_1.
    \end{equation}
    Thus, $\tmu_{t,q}(K)>1-\epsilon$ and so the collection of Gibbs states $\{\tmu_{t,q}\}_{\substack{q\leq q_1\\0\leq t\leq h}}$ is tight as desired.
\end{proof}

The previous proposition together with the following theorem tells us that if $q_n\to -\infty$ as $n\to\infty$, then the sequence of Gibbs states $\{\mu_{t,q_n}\}_{n\in\NN}$ has an accumulation point $\mu_{\infty}$.

\begin{theorem}[Prohorov's Theorem, see \cite{billingsleyWeakConvergenceMetric1968} Theorem 6.1] If a family of probability measures $\Pi$ is tight, then it is relatively compact. 
\end{theorem}

\subsection{Minimizing \textit{F}-exponents}
\noindent Now we show that $\xi_{\min}$ is achieved by $\mu_{\infty}$ independent of the $t$ used, at least on the multifractal region of $\Phi$ with respect to $F$ where such an analysis matters for our purposes. We must also work with the comparability assumption of Lemma \ref{lem:Dopen}.
\begin{customprop}{\ref{MainResult3}}
      Given the comparability of $F$ and $\log\Phi'$ from Lemma \ref{lem:Dopen}. For any $0\leq t\leq h$ fixed and any strictly decreasing sequence $\{q_n\}_{n\in\NN}$ for which $q_n\to -\infty$ as $n\to\infty$, there is a weak* accumulation point $\mu_{\infty}$ of the sequence of Gibbs states $\{\tmu_{t,q_n}\}_{n\in\NN}$, and this measure is $\tilde{f}$-minimizing, i.e. $\int_{E_A^{\infty}}\tilde{f}d\mu_{\infty} = \xi_{\min}$.
\end{customprop}
\begin{proof} Our proof is similar to \cite[Theorem 1]{jenkinsonZeroTemperatureLimits2005}, in fact if $t=0$ then the argument is nearly identical. However, the case of $t>0$ requires more care. We wish to analyze $h(\tmu_{t,q})$ for $(t,q)\in D_0$, $t$ fixed. To approach it, we start by setting $p(q):= P(t,q)$. Then $p(q)$ is real analytic for $q<q_0(t)$ where $(t,q_0(t))\in\partial D$,
\[
    p'(q) = \int_{E_A^{\infty}} \tilde{f}\;d\tmu_{t,q},
\]
and $p$ is convex. So $q\mapsto p'(q)$ is non-decreasing and bounded below by $\inf \tilde{f} \geq K\geq 0$ so 
\[
    \lim_{q\to -\infty } p'(q) = \lim_{q\to-\infty} \int_{E_A^{\infty}}\tilde{f}\; d\tmu_{t,q}\geq K > -\infty.
\]
Thus, the limit exists and is finite. Now since $\tmu_{t,q}$ is an equilibrium state
\begin{align}\label{ineq:VPineq1}
    &\int_{E_A^{\infty}} f_{t,q}d\tmu_{t,q} +h(\tmu_{t,q}) \geq \int_{E_A^{\infty}} f_{t,q}d\mu_{\infty} +h(\mu_{\infty}) \nonumber \\
    &\Rightarrow t\int_{E_A^{\infty}} \Xip d\tmu_{t,q} + q\int_{E_A^{\infty}} \tilde{f}d\tmu_{t,q} + h(\tmu_{t,q}) \geq t\int_{E_A^{\infty}} \Xip d\mu_{\infty} + q\int_{E_A^{\infty}} \tilde{f}d\mu_{\infty}  +h(\mu_{\infty}).
\end{align}
Now since $(t,q)\in D_0$, we have that $t\geq 0$ and $\int_{E_A^{\infty}}\Xip d\tmu_{t,q}<0$. So, from these observations and (\ref{ineq:VPineq1}), we yield for $q$ large and negative
\begin{align}\label{ineq:VPineq2}
    &q\int_{E_A^{\infty}} \tilde{f}d\tmu_{t,q} + h(\tmu_{t,q}) \geq t\int_{E_A^{\infty}} \Xip d\mu_{\infty} + q\int_{E_A^{\infty}} \tilde{f}d\mu_{\infty}  +h(\mu_{\infty})\nonumber\\
    &\Rightarrow \int_{E_A^{\infty}} \tilde{f}d\tmu_{t,q} + \frac{h(\tmu_{t,q})}{q} \leq \frac{t}{q}\int_{E_A^{\infty}} \Xip d\mu_{\infty} + \int_{E_A^{\infty}} \tilde{f}d\mu_{\infty}  +\frac{h(\mu_{\infty})}{q}.
\end{align}
By the comparability of $F$ and $\log\Phi'$, we have that 
\begin{align} \label{ineq:FPhiIntcomp}
    \frac{f_i-\gamma}{-\alpha}\leq \log|\phi_i'| \implies -t\int_{E_A^{\infty}}\Xip d\tmu_{t,q} \leq \frac{t}{\alpha}\int_{E_A^{\infty}}(\tilde{f}-\gamma)d\tmu_{t,q}.
\end{align}
By applying the variational principle with (\ref{ineq:FPhiIntcomp}) we have that
\begin{align}\label{ineq:VPDefofk}
    0\leq h(\tmu_{t,q}) &= P(t,q)-t\int_{E_A^{\infty}}\Xip d\tmu_{t,q}-q\int_{E_A^{\infty}}\tilde{f}d \tmu_{t,q}\nonumber\\
    & \leq P(t,q)+\frac{t}{\alpha}\int_{E_A^{\infty}}(\tilde{f}-\gamma)d\tmu_{t,q}-q\int_{E_A^{\infty}}\tilde{f}d \tmu_{t,q}\nonumber\\
    &= P(t,q) + \left(\frac{t}{\alpha}-q\right)\int_{E_A^{\infty}}\tilde{f}d \tmu_{t,q}- \frac{t\gamma}{\alpha} := k(q).
\end{align}
So $0\leq h(\tmu_{t,q})\leq k(q)$ and differentiating $k(q)$ we yield
\begin{align}\label{deri:kofq}
    k'(q) = \frac{\partial P}{\partial q}(t,q)-\frac{\partial P}{\partial q}(t,q) +\left(\frac{t}{\alpha}-q\right)\frac{\partial^2 P}{\partial q^2}(t,q) =  \left(\frac{t}{\alpha}-q\right)\frac{\partial^2 P}{\partial q^2}(t,q).
\end{align}

So for $q$ large and negative $k'(q) > 0$, as $q\to-\infty$, $k(q)$ (moving in the negative direction) is decreasing and bounded below by $0$. So applying the squeeze theorem to $\frac{h(\tmu_{t,q})}{q}$ via (\ref{ineq:VPDefofk}) and taking the limit as $q\to -\infty$ in (\ref{ineq:VPineq2}) we yield
\begin{equation}\label{ineq:Minfty1}
    \lim_{q\to-\infty} \int_{E_A^{\infty}}fd\tmu_{t,q} \leq \int_{E_A^{\infty}}fd\mu_{\infty}.
\end{equation}
Now since $\tilde{f}$ is continuous and bounded below, $-\tilde{f}$ is continuous and bounded above, so by \cite[Lemma 1]{jenkinsonZeroTemperatureLimits2005} the map $\MM\ni\mu\mapsto\int_{E_A^{\infty}}\tilde{f} d\mu\in\RR$ is lower semi-continuous with respect to the weak* topology. So we also have 
\begin{equation}\label{ineq:Minfty2}
    \infty>\liminf_{q\to-\infty}\int_{E_A^{\infty}}\tilde{f}d\tmu_{t,q} = \lim_{q\to-\infty}\int_{E_A^{\infty}}\tilde{f}d\tmu_{t,q} \geq  \int_{E_A^{\infty}}\tilde{f} d\mu_{\infty}.
\end{equation}
Combining (\ref{ineq:Minfty1}) and (\ref{ineq:Minfty2}) yields that 
\[
    \lim_{q\to-\infty} \int_{E_A^{\infty}} \tilde{f} d\tmu_{t,q} =  \int_{E_A^{\infty}} \tilde{f} d\mu_{\infty}.
\]
So the limit exists and is finite. Now, we need only check that $\mu_{\infty}$ is $\tilde{f}$-minimizing. Suppose by way of contradiction that $\mu_{\infty}$ is not $\tilde{f}$-minimizing. Then there exists $\nu\in \MM$ that achieves a smaller $F$-exponent. In particular,
\begin{equation}\label{ineq:minimizing1}
    \int_{E_A^{\infty}}\tilde{f} d\nu - \int_{E_A^{\infty}}\tilde{f}d\mu_{\infty} = -\varepsilon<0.
\end{equation}
 Note that since $\tilde{f}$ is bounded below and integrable, for some $q<0$ we have that $0 > \int_{E_A^{\infty}}q\tilde{f}d\nu >-\infty$. Further, we may find $q<0$ such that $(0,q)\in D$, so,  $P(q\tilde{f}) <\infty$ as well. So, by the variational principle, and since $\tmu_{0,q}$ is an equilibrium state, $h(\nu)<\infty$. Define the following affine map
 \begin{equation}\label{eqn:AffineMap1}
        \ell_\nu(q) = h(\nu)+\int_{E_A^{\infty}}f_{t,q}d\nu.
 \end{equation}
 Then, as $q\to-\infty$, the map $(-\infty,q_0(t))\ni q\mapsto p'(q) = \int_{E_A^{\infty}}\tilde{f} \tmu_{t,q}\in\RR$ is a function which decreases to its limit. So,
 \begin{equation}\label{eqn:Minfty1}
        \int_{E_A^{\infty}}\tilde{f}d\mu_{\infty}\leq \int_{E_A^{\infty}}\tilde{f}d\tmu_{t,q} = p'(q) \text{ for all } q<q_0(t).
 \end{equation}
 Taking the derivative of (\ref{eqn:AffineMap1}) with respect to $q$ in conjunction with (\ref{ineq:minimizing1}) gives us
 \begin{equation}\label{ineq:AffineMap1}
    \ell_{\nu}'(q) = \frac{d}{dq}\left(t\int_{E_A^{\infty}}\Xip d\nu+q\int_{E_A^{\infty}}\tilde{f}d\nu\right) = \int_{E_A^{\infty}}\tilde{f}d\nu = -\varepsilon +\int_{E_A^{\infty}}\tilde{f}d\mu_{\infty}\leq -\varepsilon +\frac{\partial P}{\partial q}(t,q).
 \end{equation}
 However, (\ref{ineq:AffineMap1}) means that for large negative $q$ we have that $\ell_{\nu}(q)>P(t,q)$ which means
 \[
    h(\nu)+\int_{E_A^{\infty}}f_{t,q}d\nu>P(t,q), 
 \]
which contradicts the variational principal. \hfill $\lightning$\newline
Hence, $\mu_{\infty}$ must be $\tilde{f}$-minimizing. Hence, there exists $\xi_{\min}$ independent of $0\leq t\leq h$ such that $\lim_{q\to-\infty} \int_{E_A^{\infty}}\tilde{f}d\tmu_{t,q} = \xi_{\min}$ as desired.

\end{proof}

\begin{remark}
    The case where $F$ is bounded, is much simpler. In this case, $\tilde{f}$ is a bounded continuous function. Thus, $\displaystyle\lim_{q\to-\infty} \int_{E_A^{\infty}}\tilde{f}d\tmu_{t,q} = \int_{E_A^{\infty}}\tilde{f}d\tmu_{\infty}$ as $\tmu_{\infty}$ is a weak* accumulation point. The argument for $\tilde{f}$-minimizing is identical to the comparable $F$ case. 
\end{remark}

The above argument results in the following immediate corollary which will be helpful for the multifractal analysis.
\begin{corollary}\label{cor:qderivativelimit}
For any $0\leq t\leq h$ fixed, when $F$ and $\log\Phi'$ are comparable or bounded as in Lemma \ref{lem:Dopen} or Lemma \ref{lem:FBoundedDopen}, respectively we have that
      \[
      \lim_{q\to -\infty} \frac{\partial P}{\partial q}(t,q) = \xi_{\min}.
      \]
\end{corollary}
\begin{remark}\label{rmk:0qSequenceIsfMinimizing}
    Without the comparability or bounded assumptions, we note that the argument above shows that the measure $\mu_{0,\infty} := \displaystyle\lim_{q\to-\infty}\tmu_{0,q}$ is $\tilde{f}$ minimizing.
\end{remark}

\section{Multifractal Analysis of \textit{F}-exponents}
Our multifractal analysis of CGDMSs concerns the decomposition of $J$ derived from $\Phi$ satisfying the SOSC, which are cofinitely regular and finitely irreducible. The decomposition is constructed by Birkhoff averages determined by a H\"older family of functions $F$ that are uniformly bounded below. We will retrieve the spectrum by applying the SOSC to show that the Hausdorff dimensions of sets determined by the projection of the decomposition at the symbolic level agree with the Hausdorff dimensions of the decomposition formed on the level set. We define $J_{>1}$ as the set of points in $J$ that have non-unique codings in the symbolic representation $E_A^{\infty}$. The multifractal decomposition of $E_A^{\infty}$ by Birkhoff averages with respect to the amalgamated function $\tilde{f}$ of $F$ is given by level sets
\begin{equation}
     E_A^{\infty}(\xi) = \left\{\om\in E_A^{\infty} \;\Big|\; \lim_{n\to\infty}\frac{1}{n}\sum_{j=0}^{n-1} \tilde{f}(\sigma^j(\om)) = \xi \right\}.
\end{equation}
We let $E_A^{\infty'}$ denote the exceptional set where the limit does not exist. The projected level sets $J(\xi):=\pi(E_A^{\infty}(\xi))$ and the projected exceptional set $J':=\pi(E_A^{\infty'})$ may not partition $J$. We call the collection $\{J(\xi)\}_{\xi\in \RR}\cup\{J'\}$ the projected multifractal decomposition of $E_A^{\infty}$ by Birkhoff averages with respect to $\tilde{f}$. By definition, $J(\xi)$ is the set of points in $J$ for which there exists a coding $\om$ which achieves the $F$-exponent $\xi$. Alternatively, consider $x=\pi(\om)\in J\setminus J_{>1}$. Notice that 
\begin{equation}
    \pi(\sigma(\om)) = \phi_{\om_1}^{-1}\left(\bigcap_{n=1}^{\infty}\phi_{\om\vert_n}(X_{t(\om\vert_n)})\right).
\end{equation}
And so, 
\begin{equation}
    \tilde{f}(\sigma^j(\om)) = f_{\sigma^j(\om)_1}(\pi(\sigma^{j+1}(\om))) = f_{\sigma^j(\om)_1}\circ \phi^{-1}_{\om\vert_{j+1}}(x).
\end{equation}
Thus, to form a proper multifractal decomposition on the limit set, we take 
\begin{equation}
    J_1(\xi) = \left\{x\in J\setminus J_{>1} \;\Big|\; \lim_{n\to\infty}\frac{1}{n}\sum_{j=0}^{n-1} f_{\sigma^j(\om)_1}\circ \phi^{-1}_{\om\vert_j}(x) = \xi \right\}
\end{equation}
and let $J_1'$ be the exceptional set which is the collection of points where these limit do not exist.
Since $x\in J\setminus J_{>1}$ implies that $x = \pi(\om)$ for some unique $\om\in E_A^{\infty}$, the use of $\om$ in the definition of $J_1(\xi)$ is not ambiguous. We call the non-empty sets in the collection $\{J_1(\xi)\}_{\xi\in\mathbb{R}}\cup\{J_1'\cup J_{>1}\}$ the multifractal decomposition of $J$ by Birkhoff averages  with respect to $F$. These sets are well-defined and partition $J$.
We can see that the measures $\mu_{t,q}$ for $(t,q)\in D$ are supported on a $J(\xi)$ for some $\xi$ since by Birkhoff's Ergodic Theorem, $\int_{E_A^{\infty}} \tilde{f}d\tmu_{t,q} = \xi$ for some $\xi$ and this integral is equal to the Birkhoff average for $\tmu_{t,q}$-a.e. $\om\in E_A^{\infty}$. And so,
\begin{equation}
    \tmu_{t,q}(E_A^{\infty}(\xi)) = 1 \implies \mu_{t,q}(J(\xi)) = \tmu_{t,q}\circ\pi^{-1}(\pi(E_A^{\infty}(\xi)))\geq \tmu_{t,q}(E_A^{\infty}(\xi)) = 1.
\end{equation}
Hence, 
\begin{equation}\label{eqn:support}
        \mu_{t,q}(J(\xi))=1.
\end{equation}
Notice that $\mu_{t,q}(J_{>1})=0$ since for a point $x\in J_{>1}$ with non-unique codings $\om$ and $\tau$ there exist an $n$ and $m$ such that $\om\vert_n$ and $\tau\vert_m$ are incomparable with
\begin{equation}
    x\in \phi_{\om\vert_n}(X_{t(\om\vert_n)})\cap\phi_{\tau\vert_m}(X_{t(\tau\vert_m)})).
\end{equation}
So, by the SOSC by way of application of \cite[Theorem 19.7.2(c)]{urbanskiVolumeFinerThermodynamic2022}, $J_{>1}$ is contained in a countable union of measure zero sets and, hence, is measure $0$. Thus, we observe that 
\begin{equation}\label{eqn:supportUnique}
        \mu_{t,q}(J_1(\xi))=1.
\end{equation}

\begin{remark}
We remind the reader here that by Lemma \ref{lem:SpectrumIsTranslationOfFamilyInvariant} translating the family $F$ by a constant results in a bijection between the sets $J_F(\xi)$ and $J_{F'}(\xi-\varepsilon+K)$. Hence, projected multifractal decomposition of $E_A^{\infty}$ by Birkhoff averages with respect to $\tilde{f}$ is isomorphic to the projected multifractal decomposition of $E_A^{\infty}$ by Birkhoff averages with respect to $\tilde{f}'$, the translated family. For this reason, we may always assume, without loss of generality, that $F$ is strictly positive.
\end{remark}

\begin{remark}
If one can express $F$ as the sum of two families $F = G+H$ where $G$ is H\"older and uniformly bounded below and $H$ is H\"older and uniformly bounded, our analysis also applies. This is essentially a restatement of Remark \ref{rmk:LyaDistortion}.
\end{remark}

We will use a version of the volume lemma in the proof of our main theorem. First, recall that for a metric space $X$ and a Borel probability measure $\mu$ on $X$, we have
\begin{equation}
        \HD(\mu):= \inf\{HD(Y)\;|\; \mu(Y)=1\}.
\end{equation}
\begin{theorem}[c.f. \cite{urbanskiVolumeFinerThermodynamic2022} Theorem 19.8.30]\label{thm:VolumeLemma} Let $\Phi$ be a finitely irreducible CGDMS satisfying the SOSC and $\mu$ be a $\sigma$-invariant Borel probability measure on $E_A^{\infty}$ such that $\chi(\mu)<\infty$, then
\[
    \HD(\mu\circ\pi^{-1}) = \frac{h(\mu)}{\chi(\mu)}.
\]
\end{theorem}
\begin{remark}
    One should note that the referenced source's argument for the volume lemma uses the cone condition by way of \cite[Lemma 19.7.16]{urbanskiVolumeFinerThermodynamic2022} to produce a uniform constant $L\geq 1$. We can still produce a uniform constant by applying \cite[Lemma 19.3.12]{urbanskiVolumeFinerThermodynamic2022} with the choice of constants $\kappa_1=1$ and $\kappa_2=\diam(X)+1$ and the derivation of this constant does not require a cone condition. 
\end{remark}

\subsection{Derivation of the Multifractal Spectrum}

We are now ready to address our main result which we recall here.

\begin{customtheorem}{\ref{thm:NDThm1}}
Let $\Phi = \{\phi_e\}_{e\in\NN}$ be a cofinitely regular finitely irreducible CGDMS satisfying the SOSC. Let $F$ be a strictly positive H\"older family of potentials that are either comparable to $\log\Phi'$ (see  Lemma \ref{lem:Dopen}) or bounded. Let $D$ be the Manhattan region of the associated pressure function $P(t,q)$, and $(t,q)\in D$. If the amalgamated functions $\tilde{f}$ and $\Xip$ of $F$ and $\log\Phi'$ respectively are such that 
\begin{equation}
    \int_{E_A^{\infty}}(|\tilde{f}|+|\Xip|)d\tmu_{t,q}<\infty,
\end{equation}
where $\tmu_{t,q}$ is the unique shift invariant Borel probability measure equivalent to the unique $t\Xip+q\tilde{f}$-conformal probability measure $\tilde{m}_{t,q}$ on $E_A^{\infty}$ such that $\LL^{*}_{f_{t,q}}\tilde{m}_{t,q} = e^{P(t,q)}\tilde{m}_{t,q}$ (defined in \cite[Chapter 17.6]{urbanskiVolumeFinerThermodynamic2022}), then:
\begin{enumerate}
    \item The system of equations 
\begin{equation}\label{eqn:Diffeq1}
    \begin{cases}P(t,q) = q\xi \\ \frac{\partial P}{\partial q}(t,q) = \xi \end{cases} 
\end{equation}
has a unique solution $(t(\xi),q(\xi))\in D_0 = \{(t,q)\in D\;|\; 0\leq t\leq h\}$ for $\xi\in (\xi_{min},\infty)$ and $h$ the Bowen parameter of $\Phi$.
    \item $t(\xi)$ and $q(\xi)$ are real analytic.
    \item $t(\xi) = \HD(J(\xi))$. 
    \item $t(\xi) = \HD (J_1(\xi))$.
\end{enumerate}
\end{customtheorem}

 $t(\xi)$ gives both the Hausdorff dimension of the $J(\xi)$ component of the projected multifractal decomposition of $E_A^{\infty}$ by Birkhoff averages with respect to $\tilde{f}$ and the Hausdorff dimension of the $J_1(\xi)$ component of the multifractal decomposition of $J$ by Birkhoff averages with respect to $F$.
 We call the function $t(\xi)$ the $F$-spectrum. We begin at the end by showing that if the solution $(t(\xi),q(\xi))\in D$ exists, then $t(\xi)$ gives the Hausdorff dimensions of both $J(\xi)$ and $J_1(\xi)$, and the solution is in $D_0$. Although the analysis of Fan et al.\;\cite{fanKhintchineExponentsLyapunov2009} may be adapted directly to our scenario, we improve upon the proof by replacing their lower estimate analysis via local Markov dimension with one that relies solely on the variational principle and the volume lemma. The upper estimate analysis is analogous to the Fan et al.\;\cite{fanKhintchineExponentsLyapunov2009} argument which we recreate here in our setting for completeness.

\begin{proposition}\label{prop:SpectrumIsT}
    Suppose $\Phi = \{\phi_e\}_{e\in\NN}$ and $F$ are as in Theorem \ref{thm:NDThm1}. If the system 
    \begin{align}\label{sys:MFSys}
    \begin{cases}P(t,q) = q\xi \\ \frac{\partial P}{\partial q}(t,q) = \xi \end{cases}
    \end{align}
    has a unique real analytic solution $(t(\xi),q(\xi))\in D$ for $\xi\in (\xi_{min},\infty)$, then $t(\xi)$ gives the Hausdorff dimension of the $J(\xi)$ component of the projected $\tilde{f}$-multifractal decomposition of $E_A^{\infty}$ by Birkhoff averages. Furthermore, $t(\xi)$ gives the Hausdorff dimension of the $J_1(\xi)$ component of the $F$-multifractal decomposition of $J$ by Birkhoff averages and $(t(\xi),q(\xi))\in D_0$. In the case that $F$ is bounded we replace $\infty$ above with $\xi_{\max} := \displaystyle\sup_{\om\in E_A^{\infty}}\{\xi(\om)\}$ so that $\xi\in(\xi_{\min},\xi_{\max})$.
\end{proposition}

\begin{proof}
    For the lower estimate we first recall that $\mu_{t(\xi),q(\xi)}(J(\xi))=1$ and begin by noting that for the solution to the system $(t(\xi),q(\xi))$, the associated Gibbs state $\mu_{t(\xi),q(\xi)}$ is an equilibrium state (Theorem \ref{thm:Gibbs&EquilibriumStates}), so the variational principle (Theorem \ref{thm:VP}) applies. Further since $\Phi$ is finitely irreducible and satisfies the SOSC the volume lemma (Theorem \ref{thm:VolumeLemma}) applies. Applying these results along with the definition of the characteristic Lyapunov exponent with respect to $\mu_{t(\xi),q(\xi)}$ (Definition \ref{def:Lyapunov}) and the definition of the Hausdorff dimension of measure given in the introduction of this section gives that
    \begin{align*}
        P(t(\xi),q(\xi)) &=q(\xi)\cdot \xi = h(\mu_{t(\xi),q(\xi)})+t(\xi)\frac{\partial P}{\partial t}(t(\xi),q(\xi))+q(\xi)\frac{\partial P}{\partial q}(t(\xi),q(\xi))\\
        \implies & q(\xi)\cdot \xi = h(\mu_{t(\xi),q(\xi)})-t(\xi)\chi(\mu_{t(\xi),q(\xi)})+q(\xi)\cdot \xi\\
        \implies & t(\xi) = \frac{h(\mu_{t(\xi),q(\xi)})}{\chi(\mu_{t(\xi),q(\xi)})} = \HD(\mu_{t(\xi),q(\xi)}) \leq \HD(J(\xi)).
    \end{align*}
    For the other inequality, we first consider the case of $q(\xi)\neq 0$. The case of $q(\xi)=0$ will be addressed at the end of the proof. We start by letting $(t,q)\in D$ and $\tmu_{t,q} := \mu_{f_{t,q}}$ be the corresponding Gibbs measure on $E_A^{\infty}$ where $\tmu_{t,q}=\mu_{t,q}\circ\pi^{-1}$ with $\mu_{t,q}$ on $J$. Then the Gibbs property grants that 
    \begin{align}\label{eqn:Gibbs1}
        \tmu_{t,q}([\om\vert_{n}])\asymp \exp(-nP(t,q))\cdot\|\phi'_{\om\vert_n}\|^t_{X_{t(\om\vert_n)}} \prod_{j=1}^{n} \exp(f_{\om_j}(\pi(\sigma^j(\om))))^q.
    \end{align}
    By applications of the bounded distortion property \cite[Lemmas 19.3.9 and 19.3.11]{urbanskiVolumeFinerThermodynamic2022} we have that 
    \begin{align} \label{eqn:Asymp1}
        \diam(\phi_{\om\vert_n}(X_{t(\om\vert_n)}))\asymp \|\phi_{\om\vert_n}'\|_{X_{t(\om\vert_n)}}.
    \end{align}
    Then combining (\ref{eqn:Gibbs1}) and (\ref{eqn:Asymp1}) one obtains
    \begin{align}\label{eqn:Gibbs2}
         \tmu_{t,q}([\om\vert_{n}]) \asymp \exp(-nP(t,q))\cdot\diam(\phi_{\om\vert_n}(X_{t(\om\vert_n)}))^t_{X_{t(\om\vert_n)}} \prod_{j=1}^{n} \exp(f_{\om_j}(\pi(\sigma^j(\om))))^q.
    \end{align}
    Now, for any $t>t(\xi)$, we consider an $\varepsilon_0>0$ such that
    \begin{equation} \label{ineq:q>0}
        0<\varepsilon_0<\frac{P(t(\xi),q(\xi))-P(t,q(\xi))}{q(\xi)} \hspace{.5cm} \text{ when } q(\xi)>0
    \end{equation}
    and 
    \begin{equation} \label{ineq:q<0}
        0<\varepsilon_0<\frac{P(t,q(\xi))-P(t(\xi),q(\xi))}{q(\xi)} \hspace{.5cm} \text{ when } q(\xi)<0.
    \end{equation}
    Such an $\varepsilon_0$ exists since $\frac{\partial P}{\partial t}<0$. Now for $n\geq 1$ we set
    \begin{align*}
        J_{\xi}^n(\varepsilon_0)&:= \pi\left(\left\{\om\in E_{A}^{\infty}\;\Bigg|\;\xi-\varepsilon_0 <\frac{1}{n}\sum_{j=1}^{n}\tilde{f}(\sigma^j(\om))<\xi+\varepsilon_0\right\}\right).
    \end{align*}
    Then $J(\xi)\subset \displaystyle \bigcup_{N=1}^{\infty}\bigcap_{n=N}^{\infty} J_{\xi}^n(\varepsilon_0)$. Now let $\II(n,\xi,\varepsilon_0)$ be the collection of $n^{th}$ level cylinders $[w_1\dots w_n]$ such that 
    \begin{equation}
        \xi -\varepsilon_0<\frac{1}{n}\sum_{j=1}^{n}\tilde{f}(\sigma^j(\om))<\xi+\varepsilon_0 \hspace{.5cm} \text{ when } \om\in [w_1\dots w_n].
    \end{equation}
    $\II(n,\xi,\varepsilon_0)$ is non-empty for $n$ large enough. To see this, note that $\tilde{f}$ is H\"older continuous on cylinders. Then by the triangle inequality for $\om,\tau\in [w_1\dots w_n]$,
    \begin{equation}
        \left|\frac{1}{n}\sum_{j=1}^{n}\tilde{f}(\sigma^j(\om))-\frac{1}{n}\sum_{j=1}^{n}\tilde{f}(\sigma^j(\tau))\right|\leq \frac{1}{n}\sum_{j=1}^{n}c\left(k^{n-j}\right)^\beta < \frac{c}{n}\cdot\frac{k^{\beta}}{1-k^{\beta}}.
    \end{equation}
    Notice that $J_{\xi}^n(\varepsilon_0) = \displaystyle\bigcup_{K\in \II(n,\xi,\varepsilon_0)}\pi(K)$ and so $\{\pi(K)\:|\: K\in \II(n,\xi,\varepsilon_0) \text{ for } n\geq 1\}$ covers $J_{\xi}^n(\varepsilon_0)$. We note the following multiplication by $1$ which we will analyze to estimate the Hausdorff Dimension of $J(\xi)$.
    \begin{equation}\label{eqn:HDestimate1}
         \diam(\pi(K))^t = \frac{e^{nP(t,q(\xi))}}{\displaystyle\prod_{j=1}^n \exp(f_{\om_j}(\pi(\sigma^j(\om))))^q}\cdot\frac{\diam(\pi(K))^t\displaystyle\prod_{j=1}^n \exp(f_{\om_j}(\pi(\sigma^j(\om))))^q}{e^{nP(t,q(\xi))}}.
    \end{equation}
    We also note that for $\om \in K\in \II(n,\xi,\varepsilon_0)$, if $q(\xi)>0$
    \begin{equation}\label{ineq:CoverIneq1}
        n(\xi-\varepsilon_0)<\sum_{j=1}^{n}\tilde{f}(\sigma^j(\om)) \implies \frac{1}{e^{nq(\xi)(\xi-\varepsilon_0)}}>\frac{1}{\displaystyle\prod_{j=1}^{n}\exp(\tilde{f}(\sigma^j(\om)))^{q(\xi)}},
    \end{equation}
    and if $q(\xi)<0$,
    \begin{equation}\label{ineq:CoverIneq2}
        n(\xi+\varepsilon_0)>\sum_{j=1}^{n}\tilde{f}(\sigma^j(\om)) \implies \frac{1}{e^{nq(\xi)(\xi+\varepsilon_0)}}>\frac{1}{\displaystyle\prod_{j=1}^{n}\exp(\tilde{f}(\sigma^j(\om)))^{q(\xi)}}.
    \end{equation}
    Now we have the tools to complete the dimension estimate. We first focus on the case of $q(\xi)>0$. By choice of $\varepsilon_0$, (\ref{ineq:q>0}), and since $(t(\xi),q(\xi))$ solves (\ref{eqn:Diffeq1}), we have that 
    \begin{align}\label{ineq:IsGeo1}
        &\varepsilon_0q(\xi)<P(t(\xi),q(\xi))-P(t,q(\xi))\nonumber\\
        &\Rightarrow P(t,q(\xi))+\varepsilon_0q(\xi)<P(t(\xi),q(\xi))\nonumber\\
        &\Rightarrow P(t,q(\xi))+\varepsilon_0q(\xi)-q(\xi)\xi<P(t(\xi),q(\xi))-q(\xi)\xi = 0.
    \end{align}
    And so, applying the Gibbs property, (\ref{eqn:Gibbs2}), along with (\ref{ineq:CoverIneq1}) and (\ref{ineq:IsGeo1}) to (\ref{eqn:HDestimate1}) we have that 
    \begin{equation} \sum_{n=1}^{\infty}\sum_{K\in\II(n,\xi,\varepsilon_o)}\diam(\pi(K))^t \leq C\sum_{n=1}^{\infty}e^{n(P(t,q(\xi))-q(\xi)(\xi-\varepsilon_0))}\sum_{K\in\II(n,\xi,\varepsilon_o)}\tmu_{t,q}(K)<\infty.
    \end{equation}
    Next, we focus on the case of $q(\xi)<0$. By choice of $\varepsilon_0$, (\ref{ineq:q<0}), and since $(t(\xi),q(\xi))$ solves (\ref{eqn:Diffeq1}), we have that 
    \begin{align}\label{ineq:IsGeo2}
        &\varepsilon_0q(\xi)>P(t,q(\xi))-P(t(\xi),q(\xi))\nonumber\\
        &\Rightarrow P(t,q(\xi))-\varepsilon_0q(\xi)<P(t(\xi),q(\xi))\nonumber\\
        &\Rightarrow P(t,q(\xi))-\varepsilon_0q(\xi)-q(\xi)\xi<P(t(\xi),q(\xi))-q(\xi)\xi = 0.
    \end{align}
    And so, applying the Gibbs property, (\ref{eqn:Gibbs2}), along with (\ref{ineq:CoverIneq2}) and (\ref{ineq:IsGeo2}) to (\ref{eqn:HDestimate1}) we have that 
    \begin{equation} \sum_{n=1}^{\infty}\sum_{K\in\II(n,\xi,\varepsilon_o)}\diam(\pi(K))^t \leq C\sum_{n=1}^{\infty}e^{n(P(t,q(\xi))-q(\xi)(\xi+\varepsilon_0))}\sum_{K\in\II(n,\xi,\varepsilon_o)}\tmu_{t,q}(K)<\infty.
    \end{equation}
        For all $K\in\II(n,\xi,\varepsilon_o)$, $\diam(\pi(K))$ does not exceed $s^n\max\{\diam(X_{t(v)})\;|\;v\in V\}$. So, for a given $\delta>0$, one may take $n$ large enough such that $\{\pi(K)\}_{K\in\II(n,\xi,\varepsilon_o)}$ forms a $\delta$-cover of $J$. Thus, the above estimates show for $q(\xi)\neq 0$ that $t(\xi)\leq HD(J(\xi))\leq t(\xi)$, which implies that $\HD(J(\xi))=t(\xi)$. And so, $0\leq t(\xi)\leq h$, where $h$ is the Bowen parameter of $\Phi$ so $(t(\xi),q(\xi))\in D_0$ when $q(\xi)\neq 0$. We must now consider the case of $q(\xi)=0$ to complete the proof of the claim. Notice that in the case of $q(\xi)=0$ the $F$-exponent is $\xi$, $\mu_{t(\xi),0}$-a.e. and since $(t(\xi),q(\xi))\in D$ we have
    \begin{equation}  
        0 = q(\xi)\xi = P(t(\xi),q(\xi))= P(t(\xi),0) = P(t(\xi)) = \lim_{n\to\infty}\frac{1}{n}Z_n(t(\xi)).
    \end{equation}
    Further, since $\Phi$ is a finitely irreducible CGDMS, by Bowen's Formula \cite[Theorem 19.6.4]{urbanskiVolumeFinerThermodynamic2022} and regularity of $\Phi$, $P(t(\xi)) = 0$ implies that
    \begin{equation}
        t(\xi) = h = HD(J) = \sup\{\HD(J_F)\;|\; F\subseteq E, \#F<\infty\}.
    \end{equation}
     In particular, $(t(\xi),0)\in D_0$ and $t(\xi) = h = HD(J)$. From our first estimation we already have that $t(\xi) = HD(\mu_{t(\xi),0})\leq HD(J(\xi))$. And since $J(\xi)\subseteq J$, we also have that $HD(J(\xi))\leq HD(J) = t(\xi)$. This completes the Hausdorff dimension analysis, and so it must be the case that $HD(J(\xi)) = t(\xi)$ for all possible $q(\xi)$ as desired. To show the equivalence of spectra between the projected decomposition and the decomposition on the limit set, we recall that since $J_1(\xi) = J(\xi)\setminus J_{>1}$ and $\mu_{t(\xi),q(\xi)}(J_1(\xi))=1$, the argument above actually shows that 
    \begin{equation}
        t(\xi) = \HD(\mu_{t(\xi),q(\xi)})\leq \HD(J_1(\xi))\leq \HD(J(\xi))= t(\xi)
    \end{equation}
    for all achievable $F$-exponents $\xi$, and so the spectrum function $t(\xi)$ for the projected multifractal decomposition of $E_A^{\infty}$ with respect to $\tilde{f}$ is identical to the spectrum function for the multifractal decomposition of $J$ with respect to $F$. 
    \end{proof}
    \noindent 
    The above proof does not require that the alphabet is infinite. This observation leads to the following corollary.
    \begin{corollary}
        If $\Phi$ is finite, irreducible, and satisfies the SOSC, and if the system (\ref{eqn:Diffeq1}) has unique solution $(t(\xi),q(\xi))\in D$ for $\xi\in(\xi_{\min},\xi_{\max})$, then $HD(J(\xi)) = HD(J_1(\xi))= t(\xi)$ and further $(t(\xi),q(\xi))\in D_0$.
    \end{corollary}
    We now turn our attention to the system of equations in question, (\ref{eqn:Diffeq1}).
     \begin{proposition} \label{prop:Existence&Uniqueness}
    Suppose that $\Phi = \{\phi_e\}_{e\in\NN}$ is as in Theorem \ref{thm:NDThm1}. If $F$ is comparable to $\log\Phi'$, the system (\ref{eqn:Diffeq1}) has a unique real analytic solution $(t(\xi),q(\xi))\in D$ for $\xi\in(\xi_{\min},\infty)$. 
    \end{proposition}

    \begin{proof}
        Let $t\in \RR$ be fixed. By Corollary \ref{cor:qderivativelimit} $\displaystyle\lim_{q\to-\infty}\frac{\partial P}{\partial q}(t,q) = \xi_{\min}$ and by Lemma \ref{lem:PDBoundaryLimits} $\displaystyle\lim_{q\to\infty}\frac{\partial P}{\partial q}(t,q) = \infty$. Note that $\frac{\partial P}{\partial q}(t,\cdot)>0$ is continuous and differentiable so for $\xi\in (\xi_{\min},\infty)$ there exists $q(t,\xi)\in (-\infty,q_0)$ given by the intermediate value theorem where $(t,q_0)\in \partial D$ such that 
        \begin{equation}
                \frac{\partial P}{\partial q}(t,q(t,\xi)) = \xi.
        \end{equation}
        This $q(t,\xi)$ is unique by monotonicity of $\frac{\partial P}{\partial q}$. Since $P(t,q)$ is analytic, the implicit $q(t,\xi)$ is also analytic with respect to $t$ and $\xi$. We set
        \begin{equation}
            W(t,\xi):= P(t,q(t,\xi))-\xi q(t,\xi).
        \end{equation}
        Then we have that 
        \begin{align}
            \frac{\partial W}{\partial t}(t,\xi) & = \frac{\partial P}{\partial t}(t,q(t,\xi))+\frac{\partial P}{\partial q}(t,q(t,\xi))\cdot\frac{\partial q}{\partial t}(t,\xi)-\xi\frac{\partial q}{\partial t}(t,\xi)\nonumber
            \\
            &= \frac{\partial P}{\partial t}(t,q(t,\xi))<0.
        \end{align}
        So $W(t,\xi)$ is strictly decreasing in $t$. Consider now a fixed $t$, then by definition of $q(t,\xi)$ we have that 
        \begin{equation} \label{eqn:ProblemChild1}
            \lim_{\xi\to\infty}\frac{\partial P}{\partial q}(t,q(t,\xi)) = \lim_{\xi\to\infty} \xi = \infty.
        \end{equation}
        Yet by the strict convexity of $P$, $\frac{\partial^2 P}{\partial q^2}>0$. So, there exists a sequence $\{\xi_n\}_{n\in\NN}\subset (\xi_{\min},\infty)$  with $\lim_{n\to\infty} \xi_n = \infty$ such that $\{q(t,\xi_n)\}_{n\in\NN}$ is a strictly increasing sequence bounded above by $q_0$. Hence, by Lemma \ref{lem:PDBoundaryLimits},
        \begin{equation}
            \lim_{n\to\infty} (t,q(t,\xi_n)) = (t,q_0)\in \partial D \implies  \lim_{n\to\infty} P(t,q(t,\xi_n)) = P(t,q_0) =\infty.
        \end{equation}
        Alternatively, if the above implication was false, then the limit of the sequence $\{(t,q_n)\}_{n\in\NN}$ would converge to some $(t,q')\in D$ with $q'<q_0$ and $\frac{\partial P}{\partial q}(t,q')= \infty$ which contradicts that $P(t,\cdot)$ is analytic on $D$. Thus, there must be an $n$ such that
        \begin{equation}
            P(t,q(t,\xi_n))>0.
        \end{equation}
        In particular, for $t=0$, there exists $\xi$ such that $P(0,q(0,\xi))>0$. $F$ is strictly positive so $\xi_{\min}>0$, so as $\xi\in (\xi_{\min},\infty)$, we have that $\xi>0$. Fix this $\xi$ and consider $W(t):= W(t,\xi)$, Then $W(t)$ is strictly decreasing. Note as well that since $(0,q(0,\xi))\in D$, we have that 
            \begin{align}\label{eqn:ProblemChild2}
            \Zt_1(0,q(0,\xi)) = \sum_{i=1}^{\infty}\|\exp(f_i)\|_{X_{t(i)}}^{q(0,\xi)}<\infty \implies q(0,\xi)<0,
        \end{align}
        and thus
        \begin{equation}\label{eqn:W0}
            W(0) = P(0,q(0,\xi))-\xi q(0,\xi)>0.
        \end{equation}
        Next, we consider the Bowen parameter $h$ of the CGDMS $\Phi$. By definition of the Bowen parameter $P(h,0)= P(h)\leq 0$. By strict convexity of $P(h,\cdot)$, if $q(h,\xi)>0$ is a solution to $P(h,q) = \xi q$,
        \begin{equation}\label{ineq:P<xiq1}
            \frac{P(h,q(h,\xi))-0}{q(h,\xi)-0} < \frac{\partial P}{\partial q}(h,q(h,\xi))=\xi \implies P(h,q(h,\xi)) < \xi q(h,\xi), \hspace{.5cm}
        \end{equation}
        and if $q(h,\xi)<0$ is a solution to $P(h,q(h,\xi)) = \xi q$
        \begin{equation}\label{ineq:P<xiq2}
            \frac{0-P(h,q(h,\xi))}{0-q(h,\xi)} > \frac{\partial P}{\partial q}(h,q(h,\xi))=\xi \implies P(h,q(h,\xi)) < \xi q(h,\xi). \hspace{.5cm}
        \end{equation}
        So, $q(h,\xi)\neq 0$ would not solve the equation $P(h,q) = \xi q$. Our only candidate then is $q(h,\xi) = 0$ and here,
        \begin{equation}
            W(h,\xi) = P(h,q(h,\xi)) = P(h,0) = P(h)= 0.
        \end{equation}
        Hence in all cases of the fixed $\xi$, $W(h)\leq 0$ where equality is achieved at $\xi=:\xi_0$. So by \eqref{eqn:W0}, the intermediate value theorem, and since $W'(t)< 0$, there exists a unique $t=t(\xi)\in (0,h]$ such that $W(t(\xi)) = 0$, i.e.
        \begin{equation}
            P(t(\xi),q(t(\xi),\xi)) = \xi q(t(\xi),\xi).
        \end{equation}
        Hence, $D_0\ni (t(\xi),q(\xi)):=(t(\xi),q(t(\xi),\xi))$ is a solution to the system of equations given. We now consider the map
        \begin{equation}
            F = \begin{bmatrix} F_1\\ F_2 \end{bmatrix} = \begin{bmatrix} P(t,q)-q\xi \\ \frac{\partial P}{\partial q}(t,q)-\xi \end{bmatrix}.
        \end{equation}
        Then 
        \begin{equation}
            \text{Jac}(F) =: \begin{bmatrix} \frac{\partial F_1}{\partial t} & \frac{\partial F_1}{\partial q}\\ \frac{\partial F_2}{\partial t} & \frac{\partial F_2}{\partial q} \end{bmatrix} = \begin{bmatrix} \frac{\partial P}{\partial t} & \frac{\partial P}{\partial q}-\xi\\ \frac{\partial^2 P}{\partial t\partial q} & \frac{\partial^2 P}{\partial q^2} \end{bmatrix}
        \end{equation}
        has determinant at $(t(\xi),q(\xi))$ given by
        \begin{align}
            \det(\text{Jac}(F))&\vert_{(t(\xi),q(\xi))}\nonumber \\&= \frac{\partial P}{\partial t}(t(\xi),q(\xi))\cdot  \frac{\partial^2 P}{\partial q^2}(t(\xi),q(\xi)) - \left(\frac{\partial P}{\partial q}(t(\xi),q(\xi))-\xi \right)\cdot \frac{\partial^2 P}{\partial t\partial q}(t(\xi),q(\xi))\nonumber\\
            &=\frac{\partial P}{\partial t}(t(\xi),q(\xi))\cdot  \frac{\partial^2 P}{\partial q^2}(t(\xi),q(\xi))<0.
        \end{align}
        Thus, by the implicit function theorem $t(\xi)$ and $q(\xi)$ are real analytic.
    \end{proof}

    \begin{corollary}\label{cor:boundedcase}
         Suppose that $\Phi = \{\phi_e\}_{e\in\NN}$ is as in Theorem \ref{thm:NDThm1}. If $F$ is bounded, then the system (\ref{eqn:Diffeq1}) has a unique real analytic solution $(t(\xi),q(\xi))\in D$ for $\xi\in(\xi_{\min},\xi_{\max})$.
    \end{corollary}

    \begin{proof}
        The argument is nearly identical after replacing $\infty$ with $\xi_{\max}$ in the range of possible $F$-exponents, but we must address (\ref{eqn:ProblemChild1}). For a fixed $t$, if $(t,q(t,\xi_{\max}))\in D$, then by Lemma \ref{lem:FBoundedDopen}, $D$ is open and there exists $\varepsilon>0$ such that $(t,q(t,\xi_{\max})+\varepsilon)\in D$. Since $\frac{\partial^2 P}{\partial q^2}>0$ we have that
        \begin{equation}
            \xi_{\max}\geq\frac{\partial P}{\partial q}(t,q(t,\xi_{\max})+\varepsilon)> \frac{\partial P}{\partial q}(t,q(t,\xi_{\max})) = \xi_{\max},
        \end{equation}
        a contradiction. \hfill $\lightning$ \newline
        So, it must be the case that $(t,q(t,\xi_{\max}))\in \partial D$ and by applying applying Lemma \ref{lem:PDBoundaryLimits} 
        \begin{equation}
             \lim_{\xi\to \xi_{\max}}\frac{\partial P}{\partial q}(t,q(t,\xi)) = \lim_{(t,q)\to p\in \partial D}\frac{\partial P}{\partial q}(t,q(t,\xi)) = \infty. 
        \end{equation}
        The rest of the proof is identical.
    \end{proof}

    \begin{corollary}\label{cor:finitecase}
        If $\Phi$ is finite, irreducible, and satisfies the SOSC, then the system (\ref{eqn:Diffeq1}) has a unique real analytic solution $(t(\xi),q(\xi))\in D$ for $\xi\in(\xi_{\min},\xi_{\max})$.
    \end{corollary}
    \begin{proof}
        The proof of Proposition \ref{prop:Existence&Uniqueness} needs more care in this case. Even after making the changes required in Corollary \ref{cor:boundedcase} we do not have (\ref{eqn:ProblemChild2}). To address this we appeal to the variational principle, Theorem \ref{thm:VP}. Note that since $\mu_{0,q(0,\xi)}$ is an equilibrium state, the supremum is acheived and so we have that 
        \begin{align*}
            P(0,q(0,\xi)) &= h_{\mu_{0,q(0,\xi)}}(\sigma) +\int_{E_A^{\infty}}f_{0,q(0,\xi)}d\mu_{0,q(0,\xi)}\\
            &= h_{\mu_{0,q(0,\xi)}}(\sigma) + \int_{E_A^{\infty}} q(0,\xi)\cdot \tilde{f} d\mu_{0,q(0,\xi)}\\
            &= h_{\mu_{0,q(0,\xi)}}(\sigma) + q(0,\xi) \cdot \int_{E_A^{\infty}}  \tilde{f} d\mu_{0,q(0,\xi)}\\
            &= h_{\mu_{0,q(0,\xi)}}(\sigma) + q(0,\xi) \cdot \frac{\partial P}{\partial q}(0,q(0,\xi))\\
            &= h_{\mu_{0,q(0,\xi)}}(\sigma) + q(0,\xi) \cdot \xi\\
            &\implies  P(0,q(0,\xi))-q(0,\xi) \cdot \xi = h_{\mu_{0,q(0,\xi)}}(\sigma) \geq 0.
        \end{align*}
        Note that $W(t)$ is still decreasing in this case and that as $\Phi$ is finite, Bowen's formula holds by \cite[Proposition 10]{ghenciuBowenFormulaShiftGenerated2015} as well as (\ref{ineq:P<xiq1}) and (\ref{ineq:P<xiq2}). The proof is otherwise unaltered.  
    \end{proof}

    \begin{proof}[Proof of Theorem \ref{thm:NDThm1}]
        The proof follows immediately from Proposition \ref{prop:SpectrumIsT}, Proposition \ref{prop:Existence&Uniqueness}, Corollary \ref{cor:boundedcase}, and Corollary \ref{cor:finitecase}.
    \end{proof}

    \subsection{Shape of \textit{F}-Exponent Spectra}
    The main result of this section is to show that the $F$-spectrum for family of potentials $F$ uniformly bounded below has shape similar to that described in the analysis of the Khintchine exponent in \cite{fanKhintchineExponentsLyapunov2009}. In particular, it is neither strictly concave nor strictly convex for infinite alphabet systems. In the process, we will need to do some curve sketching for $q$ as well. We will reproduce the appropriate arguments from \cite{fanKhintchineExponentsLyapunov2009} with the slight alterations required and references to the comparable statements as needed. We begin by noting that \cite{fanKhintchineExponentsLyapunov2009} does much of the calculus for us since the derivations of the derivatives of $t$ and $q$ are derived from applying elementary calculus and the system of equations in (\ref{sys:MFSys}). For those interested in the finite case we recommend \cite{pesinMultifractalAnalysisBirkhoff2001}. 
    \begin{proposition}\cite[Propositions 4.15, 5.2]{fanKhintchineExponentsLyapunov2009} Let $t(\xi)$ be the spectrum function found via Theorem \ref{thm:NDThm1}. For $\xi\in(\xi_{\min},\infty)$ we have the following equalities: 
    \begin{align}\label{eqn:1stDerivativesOfT}
        t'(\xi) &= \frac{q(\xi)}{\frac{\partial P}{\partial t}(t(\xi),q(\xi))},
    \end{align}
    \begin{align}\label{eqn:2ndDerviativeOfT}
        t''(\xi) &= \frac{t'(\xi)\left(\frac{\partial^2 P}{\partial t^2}(t(\xi),q(\xi))\right)-q'(\xi)^2\left(\frac{\partial^2 P}{\partial q^2}(t(\xi),q(\xi))\right)}{-\frac{\partial P}{\partial t}(t(\xi),q(\xi))},
    \end{align}
    \begin{align}\label{eqn:1stDerivativeOfQ} 
        \text{ and }\hspace {1cm} q'(\xi) &= \frac{1-t'(\xi)\left(\frac{\partial^2 P}{\partial q\partial t}(t(\xi),q(\xi))\right)}{\frac{\partial^2 P}{\partial q^2}(t(\xi),q(\xi))}.
    \end{align}
    \end{proposition}
    The form of the $t$ derivatives above and the properties of the derivatives of $P$ tell us that we need to understand $\sgn(q(\xi))$ in order to discern $\sgn(t'(\xi))$, and then, combining some limiting properties of $t$ and $q$, conclude on $\sgn(t''(\xi))$. This description of signs will wrap up the argument of the main result in this section.
    \begin{proposition}\label{prop:sgnOfq} [c.f. \cite[Proposition 4.14]{fanKhintchineExponentsLyapunov2009}]
    For $\xi\in (\xi_{\min},\infty)$ and $t(\xi_0) = h$ where $h$ is the Bowen parameter of a finitely irreducible and cofinitely regular CGDMS $\Phi$ satisfying the SOSC, we have that 
    \begin{align*}
        &q(\xi) <0  \hspace{1cm}\text{ for } \xi< \xi_0,\\
        &q(\xi_0) = 0, \\
        &q(\xi) >0  \hspace{1cm}\text{ for } \xi>\xi_0.
    \end{align*}
    \end{proposition}
    \begin{proof} Note that the system
        \begin{align}
        \begin{cases}P(t,0) = 0 \\ \frac{\partial P}{\partial q}(t,0) = \xi \end{cases}
        \end{align}
        has a unique solution which must occur at the Bowen parameter $h$ by regularity of $\Phi$. And so, the equation $q(\xi_0) = 0$ holds. Next, from the convexity of $P$, we have that for all $q\geq 0$
        \begin{equation}
            \frac{P(h,q)}{q}\geq \frac{\partial P}{\partial q}(h,0) = \xi_0 \implies P(h,q)\geq \xi_0 q,
        \end{equation}
        and for all $q\leq 0$
        \begin{equation}
            \frac{P(h,q)}{q}\leq \frac{\partial P}{\partial q}(h,0) = \xi_0 \implies P(h,q)\geq \xi_0 q.
        \end{equation}
        Suppose now that $t(\xi)\in (0,h)$ then for $\xi>\xi_0$, if $q(\xi)\leq 0$, since $\frac{\partial P}{\partial t}<0$, we have that 
        \begin{equation}
            P(t(\xi),q(\xi))> P(h,q(\xi))\geq q(\xi)\xi_0> q(\xi)\xi. \qquad \lightning
        \end{equation}
        So it must be the case that $q(\xi)>0$ when $\xi>\xi_0$. \newline 
        Similarly for $\xi<\xi_0$, if $q(\xi)\geq 0$ we have that 
        \begin{equation}
            P(t(\xi),q(\xi))> P(h,q(\xi))\geq q(\xi)\xi_0 > q(\xi)\xi. \qquad \lightning
        \end{equation}
        So it must be the case that $q(\xi)<0$ whenever $\xi<\xi_0$, as desired.
    \end{proof}
    
    Combining Proposition \ref{prop:sgnOfq} with Lemma \ref{lem:DerivativesOfPressure} immediately yields the following
    \begin{corollary}\label{prop:sgnOft'} [c.f. \cite[Proposition 4.16]{fanKhintchineExponentsLyapunov2009}]
        For $\xi\in (\xi_{\min},\infty)$ and $t(\xi_0) = h$ where $h$ is the Bowen parameter of a finitely irreducible and cofinitely regular CGDMS $\Phi$ satisfying the SOSC,  we have that 
    \begin{align*}
        &t'(\xi) > 0  \hspace{1cm}\text{ for } \xi< \xi_0,\\
        &t'(\xi_0) = 0, \\
        &t'(\xi) <0  \hspace{1cm}\text{ for } \xi>\xi_0,
    \end{align*}
    and this holds for finite, irreducible $\Phi$ satisfying the SOSC by replacing $\infty$ with $\xi_{\max}$.
    \end{corollary}

    We turn our attention to the limiting properties of $t$ and $q$. 
   \begin{proposition}\label{prop:limofT} \textup{(ND, c.f. \cite[Proposition 4.16]{fanKhintchineExponentsLyapunov2009})}
        Let $\theta$ be the finiteness parameter of a finitely irreducible, cofinitely regular CGDMS $\Phi$ satisfying the SOSC, and suppose that $F$ and $\log\Phi'$ satisfy the comparability condition as in Lemma \ref{lem:Dopen}, then we have the following limiting behavior of $t$:
        \begin{align}
        \lim_{\xi\to\xi_{\min}}t(\xi)= \inf_{\xi_{\min}< \xi\leq \xi_0}\HD(J(\xi))  \hspace{1cm} \text{and} \hspace{1cm} \lim_{\xi\to\infty}t(\xi)=\theta
        \end{align}
        when the limits are defined.
    \end{proposition}
    \begin{proof}
        By Corollary \ref{prop:sgnOft'} and Proposition \ref{prop:Existence&Uniqueness} we obtain two analytic inverse funcitons $\xi_1(t)$ on $(\xi_{\min},\xi_0)$ which is increasing and $\xi_2(t)$ on $(\xi_0,\infty)$ which is decreasing. Since $\xi_1(t)$ is increasing and continuous, by application of Theorem \ref{thm:NDThm1} we have that 
        \[
        \lim_{\xi\to\xi_{\min}}t(\xi)= \inf_{\xi_{\min}< \xi\leq \xi_0}HD(J(\xi))
        \]
        Note that by (\ref{sys:MFSys}), we may consider the system as functions of $t$ instead of $\xi$.
        Now for $\xi_2(t)$, by the cofinite regularity of $\Phi$,
        \begin{equation}
            \xi_2(t)= \frac{P(t,q(t))}{q(t)} = \frac{\partial P}{\partial q}(t,q(t)) \geq \frac{\partial P}{\partial q}(t,0) \to \infty \text{ as } t\to \theta
        \end{equation}
        where the limiting behavior of $\frac{\partial P}{\partial q}$ follows from comparability of $F$ and $\log\Phi'$. Thus, $\lim_{\xi\to\infty} t(\xi) = \theta$ as desired.
    \end{proof}
    \begin{remark}\label{rmk:NonSepShapes}
        If instead we suppose that we have $I$, a witness to the finite irreducibility of $\Phi$ such that $f_i(X_{t(i)})\cap f_j(X_{t(j)}) = \emptyset$ for all $i,j\in E(I) := \{e\in E \;|\; e \text{ occurs in some } w\in I\}$ and that the $\sup_{i\in E(I)}(f_i(X_{t(i)}))<\inf (g_j(X_{t(j)})$ for all $g_j\in F\setminus \{f_i\in F\;|\; i\in E(I)\}$. Then $\mu_{\infty}$ will be supported on some finite alphabet irreducible subshift of $E_A^{\infty}$ by \cite[Theorem 2]{bissacotExistenceMaximizingMeasures2014}. The separation condition on the images of the functions in $F$ gives us that $\mu_{\infty}$ is supported on periodic orbits contained in that finite alphabet subshift. By comparability, there can only be finitely many such periodic points that support $\mu_{\infty}$. So, $HD(J(\xi_{\min}))=0$, and thus, $0$ is the left endpoint of the open interval defining the domain of $\xi_1(t)$.  Note that by (\ref{sys:MFSys}), considering the system as functions of $t$ instead of $\xi$, we have that
        \begin{equation}
            \xi_1(t)= \frac{P(t,q(t))}{q(t)} = \frac{\partial P}{\partial q}(t,q(t)).
        \end{equation}
        Since, from Proposition \ref{prop:sgnOfq}, $q(t)<0$, it must be the case that $P(t,q(t))<0$, and by (\ref{eqn:BoundaryDPressure}) there must exist a $q_0(t)>q(t)$ such that $P(t,q_0(t)) = 0$. Thus,
        \begin{align}
            \xi_1(t)= \frac{\partial P}{\partial q}(t,q(t))<\frac{\partial P}{\partial q}(t,q_0(t)).
        \end{align}
        Now by applying (\ref{ineq:ZSqueezeP}) and noticing that 
        \begin{equation}
            Z_1(0,q) = \sum_{i=1}^{\infty}\|\exp f_i\|^q_{X_{t(i)}},
        \end{equation}
        along with the strict positivity of $F$ and comparability to $\log\Phi'$, we have that $\lim_{i\to\infty} \|\exp f_i\|_{X_{t(i)}} = \infty $. So $Z_1(0,q)$ is strictly positive for any finite $q$. Thus, it must be the case that 
        \begin{equation}\label{eqn:q0lim}
        \displaystyle \lim_{t\to 0} q_0(t) = -\infty.
        \end{equation}
        So, by analyticity of pressure and from Remark \ref{rmk:0qSequenceIsfMinimizing}
        \begin{equation}
            \lim_{t\to 0}\frac{\partial P}{\partial q}(t,q_0(t)) = \lim_{q\to-\infty}\frac{\partial P}{\partial q}(0,q) = \xi_{\min}.
        \end{equation}
        Thus, we may conclude that 
        \begin{equation}
            \lim_{t\to 0} \xi_1(t) = \lim_{t\to 0} \frac{\partial P}{\partial q}(t,q(t)) = \xi_{\min} \implies \lim_{\xi\to\xi_{\min}}t(\xi) = 0.
        \end{equation}
    \end{remark}       
    \begin{remark}
        Without explicitly using comparability one may determine that, at the least, $\theta$ is not in the interior of the domain of $\xi_2$. If it were the case that $\theta\in\text{Int}(\dom(\xi_2))$, then by analycity of $\xi_2$, we have that $\xi_2(\theta)<\infty$. Further since $q(t)$ is analytic $q(\theta)<\infty$ and as $\xi_2$ is onto $(\xi_0,\xi_{\max})$ we also have that $0<q(\theta)$. Hence $\xi_2(\theta) = \frac{P(\theta,q(\theta))}{q(\theta)} = \infty$, which is a contradiction. \hfill $\lightning$\newline  In this case, we know that $\alpha  = \inf(\dom(\xi_2))\geq\theta$. 
    \end{remark}
    \begin{proposition}\label{prop:limofQ} [c.f. \cite[Proposition 5.1]{fanKhintchineExponentsLyapunov2009}]
        Let $q(\xi)$ be as derived in Theorem \ref{thm:NDThm1}. Suppose that $F$ and $\log\Phi'$ satisfy the comparability condition as in Lemma \ref{lem:Dopen}. Then we have the following limiting behavior of $q$: 
        \[
            \lim_{\xi\to\infty} q(\xi) = 0 \hspace{1cm} \text{and} \hspace{1cm} \lim_{\xi\to\xi_{\min}} q(\xi) = -\infty.
        \]
    \end{proposition}
    \begin{proof} 
        For $\xi>\xi_0$ we have that $q(\xi)>0$. Since $q(\xi_0) = 0$ we retrieve that $0\leq q(\xi)<q_0(\xi)$ where $(t(\xi),q_0(\xi))\in \partial D$. And so, by Proposition \ref{prop:limofT} and since $(\theta,q_0(\theta)) = (\theta,0)$, we have that $\lim_{\xi\to\infty}q_0(\xi) = 0$. The result then follows by the squeeze theorem. \newline
        The fact that $\lim_{\xi\to\xi_{\min}} q(\xi) = -\infty$ also follows from the squeeze theorem by noting $-\infty<q(\xi)<q_0(\xi)$ and (\ref{eqn:q0lim}).
    \end{proof}
    We now have enough information to arrive at the destination of this section.
    \begin{theorem}\label{thm:tshape}
        The $F$-spectra derived via Theorem \ref{thm:NDThm1} for infinite CGDMSs where $F$ and $\log\Phi'$ are comparable as in Lemma \ref{lem:Dopen} are neither strictly concave nor strictly convex. Thus, there exists a $\gamma\in [\xi_0,\infty)$ such that $t''(\gamma)>0$ and $t''(\xi_0)<0$.  
    \end{theorem}
    \begin{proof}
        Notice that for $\xi= \xi_0$ we have that $t'(\xi_0) = 0$, and so from (\ref{eqn:2ndDerviativeOfT}) and (\ref{eqn:1stDerivativeOfQ}) $q'(\xi_0)\neq 0$ and so $t''(\xi_0)<0$.\newline 
        Now since $\lim_{\xi\to\infty}q(\xi)=0$, $q(\xi_0) = 0$, and $q(\xi)$ is not identically $0$, there exists $\gamma \in (\xi_0,\infty)$ such that $q'(\xi)<0$. Since $H(t,q)$ is positive semi-definite, the calculation that $t''(\beta)>0$ is identical to that of \cite[Proof of Theorem 1.2 (4), Page 103]{fanKhintchineExponentsLyapunov2009}.
    \end{proof}
    \section{Applications}
    We take care to give concrete examples that demonstrate Theorem \ref{thm:NDThm1} along with many of the results and remarks regarding the possible $F$-spectra shapes here as well. In this section, $\log$ denotes the natural logarithm unless otherwise stated. Our first application, however, is rather general, and will be applied to another example further along in this section.   
        \subsection{Lyapunov Spectrum}
    In this particular setting, we recall from the introduction that we have that $F = -\log\Phi'$. We let $P_L(t,q)$ be our two parameter Lyapunov pressure. We have $P_L(t,q) = P(t-q)$ where $P$ is the standard one parameter Lyapunov pressure which is similar to the discussion in   \cite[Section 6]{fanKhintchineExponentsLyapunov2009}. For completeness, we provide more details on the proof of \cite[Proposition 6.4]{fanKhintchineExponentsLyapunov2009} by mirroring the argument in the exposition in \cite[Chapter 17 Section 3]{falconerFractalGeometryMathematical2014} regarding the geometric interpretation of the Legendre transform. 

\begin{proposition}\label{prop:LyapSpecForm}
    For the Lyapunov spectrum of a CGDMS satisfying the hypotheses of Theorem \ref{thm:NDThm1}, we have that 
    \begin{equation}\label{eqn:LyapunovSpectrumINF}
        t(\xi) = \frac{1}{\xi}(P(-q(\xi))-q(\xi))  = \frac{1}{\xi}\inf_{q\in\RR}\{P(-q)-q\xi\},
    \end{equation}
    where $P(q):= P_L(0,-q)$.
\end{proposition}
\begin{proof}
    Note that as discussed above $P_L(t,q) = P(t-q)$ in the Lyapunov spectrum case, and so 
    \[
        \begin{cases} P(t-q,0) = q\xi \\ \frac{\partial P}{\partial q}(t-q,0) = \xi \end{cases} \implies \begin{cases} P(-q,0) = (t+q)\xi \\ \frac{\partial P}{\partial q}(-q,0) = \xi \end{cases}.
    \]
    Now solving the first equation for $t$ and expressing $t$ and $q$ as functions of $\xi$ we have that
    \begin{equation}\label{eqn:Lya1}
        \begin{cases} t(\xi) = \frac{P(-q(\xi),0)}{\xi}-q(\xi) \\ \frac{\partial P}{\partial q}(-q(\xi),0) = \xi \end{cases}.
    \end{equation}
    Hence, we have the first equality in (\ref{eqn:LyapunovSpectrumINF}). To see the second equality, we first note that since $P(-q)$ is strictly convex, so is the function $P(-q)-q\xi$. Thus, the infimum is attained at a unique point $q:=q(\alpha)$. Now since $P(-q) = P_L(0,q)$ 
    \[
        \frac{\partial P_L}{\partial q}(0,q) = -\frac{dP}{dq}(-q).
    \]
    Differentiating $P(-q)-q\xi$ with respect to $q$ and solving for critical points yields
    \[
         -\frac{dP}{dq}(-q) = \xi \implies \frac{\partial P_L}{\partial q}(0,q) = \xi.
    \]
    Yet, per Theorem \ref{thm:NDThm1}, this only occurs when $q = q(\alpha) = q(\xi)$. Hence 
    \begin{equation}\label{eqn:LegendreTransform}
        \inf_{q\in\RR}\{P(-q)-q\xi\} = P(-q(\xi))-q(\xi)\xi.
    \end{equation}
    And so, factoring out $\frac{1}{\xi}$ from the first equation in (\ref{eqn:Lya1}) and applying (\ref{eqn:LegendreTransform}) yields the desired result.
\end{proof}

\subsection{Linearized Gauss Map and Subsystems}
The linearized Gauss map on the unit interval is a dynamical system $S:[0,1]\to [0,1]$ with inverse branches
\[
    \psi_n(x) = -\frac{1}{n(n+1)}x+\frac{1}{n}
\]
for $n\in \NN\setminus \{0\}$. The collection $\Psi = \{\psi_n\}_{n=1}^{\infty}$ gives a CGDMS which has finiteness parameter $\theta = \frac{1}{2}$. We can see that this is the case, as the Lyapunov pressure function is given by 
\[
    P_L(t) = \lim_{n\to\infty}\frac{1}{n}\log\sum_{\om\in E_A^{\infty}}\prod_{j=1}^{|\om|}(\om_j(\om_j+1))^{-t}= \log \sum_{n=1}^{\infty} \left(\frac{1}{n(n+1)}\right)^t,
\]
and for $Z_1(t)$ we have that, for $t>0$,
\[
        \zeta(2t) \geq Z_1(t)\geq \sum_{n=1}^{\infty} (n+1)^{-2t},
\]
where $\zeta$ is the Riemann zeta function. In this case, the invariant measure is the Lebesgue measure, so the calculation of characteristic exponents is more straight forward, however, $P_L(t)$ is not quickly calculable. For this reason, we leave a concrete example of an infinite CGDMS Lyapunov spectrum calculation to the next subsection of examples and we proceed here with two spectra on finite subsystems for $F$-exponents that highlight Remark \ref{rmk:NonSepShapes} and Proposition \ref{prop:limofT}. These examples are similar, but they highlight the possibility that $\xi_{\min}$, $\xi_{\max}$ or both may have associated measures supported on compact subshifts of $E_A^{\infty}$ that are more complex than just the orbit of a periodic point or a countable collection of such points.

\begin{figure}[h]
    \centering
    \includegraphics{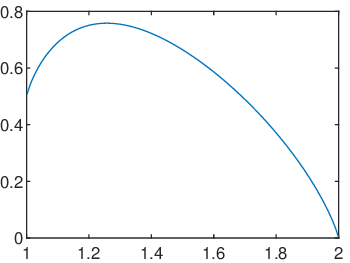}
    \caption{Graph of Numerical Solution of $t(\xi)$ for Example \ref{ex:LG3_Parity} via MATLAB}
    \label{fig:LG3_Parity}
\end{figure}

\begin{example} \label{ex:LG3_Parity}
    Let $\AL\subset E$ with $\AL=\{1,2,3\}$. We take $F = \{f_1,f_2,f_3\}$ with $f_1=f_3 = \log(e)=1$ and $f_2 = 2\log(e)=2$, where $e=2.7182...$ is Euler's number and not a symbol of the alphabet in this case. This is a strictly positive family that is a translation of the ``parity family", a collection of functions which distinguish odd symbols with $0$ from even symbols with $1$. By Lemma \ref{lem:SpectrumIsTranslationOfFamilyInvariant}, we have that the spectrum resulting from $F$ is just a translation of that of the parity family, and so, the two decompositions are isomorphic. In particular, the resulting partitions of the limit set are identical. Since the amalgamated function values only depend on the first coordinate, we have that 
    \[
        P(t,q) = \log Z_1(t,q) = \log\left(\frac{e^q}{2^t}+\frac{e^{2q}}{6^t}+\frac{e^{q}}{12^t}\right),
    \]
    and
    \begin{equation}\label{eqn:exampleDP1}
        \frac{\partial P}{\partial q}(t,q)=  \frac{(2^{-t}+12^{-t})e^q+2\cdot 6^{-t}e^{2q}}{(2^{-t}+12^{-t})e^q+6^{-t}e^{2q}} = \xi
    \end{equation}
    with
    \[
        \xi_{\min} = \lim_{q\to-\infty} \frac{\partial P}{\partial q}(t,q) = 1 \hspace{1cm} \text{and} \hspace{1cm}  \xi_{\max} = \lim_{q\to\infty} \frac{\partial P}{\partial q}(t,q) = 2.
    \]
    Solving (\ref{eqn:exampleDP1}) for $q$ yields
    \[
        q(t,\xi) = \log\left(\frac{(\xi-1)(2^{-t}+12^{-t})}{6^{-t}(2-\xi)}\right).
    \]
    The function $t(\xi)$ can be approximated by splitting the interval $(\xi_{\min},\xi_{\max})$ into $1000$ uniform sub-intervals and then solving the equation $P(t,q(t,\xi))-\xi\cdot q(t,\xi)$ at endpoints. The graph of this approximation of $t(\xi)$ is Figure \ref{fig:LG3_Parity}.
    
\end{example}

    \begin{figure}[h]
    \centering
    \includegraphics{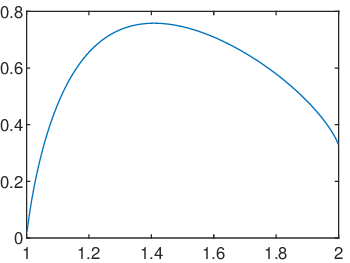}
    \caption{Graph of Numerical Solution of $t(\xi)$ for Example \ref{ex:LG3_23} via MATLAB}
    \label{fig:LG3_23}
    \end{figure}

\begin{example}\label{ex:LG3_23}
    Similar to the previous example, let $\AL\subset E$ with $\AL=\{1,2,3\}$ and instead take $F = \{f_1,f_2,f_3\}$ with $f_2=f_3 = 2\log(e)=2$ and $f_1 = \log(e)=1$,  where again $e=2.7182...$ is Euler's number and not a symbol of the alphabet.  Since the amalgamated function values only depend on the first coordinate, we have that 
    \[
        P(t,q) = \log Z_1(t,q) = \log\left(\frac{e^q}{2^t}+\frac{e^{2q}}{6^t}+\frac{e^{2q}}{12^t}\right),
    \]
    and
    \begin{equation}\label{eqn:exampleDP_2}
        \frac{\partial P}{\partial q}(t,q)=  \frac{2^{-t}e^q+2\cdot(12^{-t}+ 6^{-t})e^{2q}}{2^{-t}e^{q}+(6^{-t}+12^{-t})e^{2q}} = \xi
    \end{equation}
    with
    \[
        \xi_{\min} = \lim_{q\to-\infty} \frac{\partial P}{\partial q}(t,q) = 1 \hspace{1cm} \text{and} \hspace{1cm}  \xi_{\max} = \lim_{q\to\infty} \frac{\partial P}{\partial q}(t,q) = 2.
    \]
    Solving (\ref{eqn:exampleDP_2}) for $q$ yields
    \[
        q(t,\xi) = \log\left(\frac{(\xi-1)\cdot 2^{-t}}{(6^{-t}+12^{-t})(2-\xi)}\right).
    \]
    The function $t(\xi)$ can be approximated by splitting the interval $(\xi_{\min},\xi_{\max})$ into $1000$ uniform sub-intervals and then solving the equation $P(t,q(t,\xi))-\xi\cdot q(t,\xi)$ at endpoints. The graph of this approximation of $t(\xi)$ is Figure \ref{fig:LG3_23}.
    
\end{example}

\subsection{L\"uroth Expansions}
Our multifractal analysis also generalizes that of \cite{barreiraFrequencyDigitsLuroth2009} concerning the standard L\"uroth Expansion. The standard L\"uroth expansion was introduced in 1883 by L\"uroth \cite{lurothUeberEindeutigeEntwickelung1883} and has a collection of inverse branches that when reflected with respect to the line $y=1/2$ become the inverse branches of the linearlized Gauss map. In general, take a strictly decreasing sequence $P=\{a_k\}_{k=0}^{\infty}$ with $a_0 = 1$ and $\lim_{k\to\infty} a_k = 0$. Then a CGDMS $\Phi = \{\phi_n\}_{n=1}^{\infty}$ may be defined by taking inverse branches of the dynamical system $S_P:[0,1]\to[0,1]$ which is defined by the branches
\[
        \phi^{-1}_n(x) = \frac{x}{a_{n-1}-a_n}-\frac{a_n}{a_{n-1}-a_{n}} \text{ with } \phi^{-1}_n:[a_{n},a_{n-1}]\to[0,1] \text{ onto }.
\]
\begin{figure}[h]
    \centering
    \includegraphics{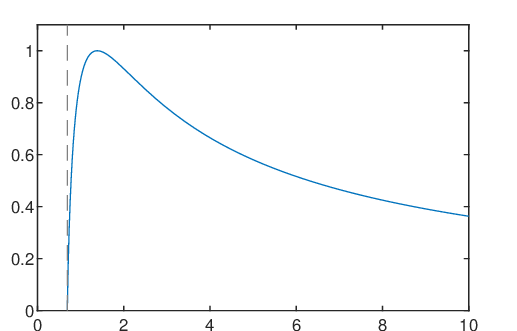}
    \caption{Graph of $t(\xi)$ for Example \ref{ex:Lu} for $\xi\in(\log 2,10))$}
    \label{fig:Lu0To10}
    \end{figure}
\begin{example}\label{ex:Lu}
    Let $\{\frac{1}{2^n}\}_{n=0}^{\infty}$ define a L\"uroth expansion. Then the geometric potential is given by $\Xip\circ\pi^{-1}(x) = -\om(x)_1\log (2)$ for any $x\in \pi(J\setminus J_{>1})$ and has Lebesgue integral given by 
    \[
        \sum_{n=1}^{\infty}\frac{-n\log2}{2^n} = -2\log 2 \approx 1.38629.
    \]
    Thus, Theorem \ref{thm:NDThm1} applies for any $F$ for which $\tilde{f}\circ\pi^{-1}$ is integrable with respect to Lebesgue measure and comparable to $\log\Phi'$. Notice that for the positive Lyapunov family $F = -\log \Phi'$ in this example we have that  
    \[
        Z_1(t,q) = \sum_{n=1}^{\infty} 2^{-n(t-q)}
    \]
    and $Z_1(t,q)$ is finite for $t-q>0$, which implies that $\theta = 0$ here. We also have that 
    \[
        P_L(t,q) = \lim_{n\to\infty}\frac{1}{n}\log\sum_{\om\in E_A^n}\prod_{j=1}^{|\om|}2^{\om_j(q-t)}  = \lim_{n\to\infty} \frac{1}{n} \log\left(\frac{2^{q-t}}{1-2^{q-t}}\right)^n = \log(2^{q-t})-\log(1-2^{q-t}) 
    \]
    and that 
    \[
        \frac{\partial P_L}{\partial q}(t,q) = \frac{2^t\log 2}{2^t-2^q}.
    \]
    Since $\xi_0 = 2\log 2$, we find that the Bowen parameter $h = t(\xi_0) =  1$ since this choice of $h$ satisfies the system of equations in Theorem $\ref{thm:NDThm1}$. From Corollary \ref{cor:qderivativelimit} we have that $\xi_{min} = \log 2$ as well. Now let $T(q) = P_L(0,q)-q\xi = P(-q)-q\xi$. Then,
    \[
        T'(q) = 0 \implies \frac{\log 2}{1-2^q}-\xi = 0 \implies q = \log_2\left(1-\frac{\log 2}{\xi}\right),
    \]
    and
    \[
        T''(q) = \left(\frac{\log 2}{1-2^q}-\xi\right)' = \frac{2^q\log^2(2)}{(1-2^q)^2}>0.
    \]
    Hence by Proposition \ref{prop:LyapSpecForm} and the $2^{nd}$ derivative test from calculus, we have that 
    \begin{align}
        t(\xi) = \frac{1}{\xi}\inf_{q\in\RR}\{P(-q)-q\xi\} &= \frac{1}{\xi}T\left(\log_2\left(1-\frac{\log 2}{\xi}\right) \right) \notag \\&= \frac{1}{\xi}\left(\log\left(1-\frac{\log(2)}{\xi}\right)-\log\left(\frac{\log(2)}{\xi}\right)\right)-\log_2\left(1-\frac{\log 2}{\xi}\right)\\
        &= \frac{1}{\xi}\log\left(\frac{\xi}{\log(2)}-1\right)-\log_2\left(1-\frac{\log 2}{\xi}\right).
    \end{align}

    \noindent The graph of $t(\xi)$ is shown in Figure \ref{fig:Lu0To10} and in the introduction of the paper in Figure \ref{fig:Lu0To101}

\end{example}

\section*{Acknowledgements}

I would like to thank my PhD Advisor Mariusz Urba\'nski for his consistent feedback and criticism and Pieter Allaart for his comments. Their feedback has undoubtedly strengthened the overall scope of the paper and its results. I would also like to thank Johannes Jaerisch and Vasilis Chousioinis for allotting time for me to talk through the results of the paper with them individually and for their feedback on potential future directions of this project. A special thanks to Jason Atnip as well for their expeditious feedback on an initial draft of the paper. This paper is a portion of the authors upcoming PhD Dissertation.

\bibliographystyle{plain}
\bibliography{NDD}
\end{document}